\documentclass[12pt,reqno]{amsart}

\usepackage{graphics}

\usepackage{latexsym}
\usepackage[dvips]{epsfig}
\usepackage{color}
\usepackage{graphicx}
\usepackage{amsmath}
\usepackage{amssymb}
\usepackage{amsfonts}
\usepackage{eufrak}
\usepackage{amstext}
\usepackage{amsopn}
\usepackage{amsbsy}
\usepackage{amscd}
\usepackage{amsxtra}
\usepackage{amsthm}
\usepackage{bm}
\usepackage{CJK}

\newcommand{\R}{{\mathbb R}}

\newcommand{\beq}{\begin{equation}}
\newcommand{\eeq}{\end{equation}}
\newcommand{\ben}{\begin{eqnarray}}
\newcommand{\een}{\end{eqnarray}}
\newcommand{\beno}{\begin{eqnarray*}}
\newcommand{\eeno}{\end{eqnarray*}}

\newtheorem{thm}{Theorem}[section]

\newtheorem{lem}[thm]{Lemma}
\newtheorem{prop}[thm]{Proposition}
\newtheorem{coro}[thm]{Corollary}
\newtheorem{rmk}[thm]{Remark}

\allowdisplaybreaks

\begin{document}

\title[]{On the uniqueness of solutions of an nonlocal elliptic system}

\author[K. Wang]{ Kelei Wang}
 \address{\noindent K. Wang-
 Wuhan Institute of Physics and Mathematics,
The Chinese Academy of Sciences, Wuhan 430071, China.}
\email{wangkelei@wipm.ac.cn}

\author[J. Wei]{Juncheng Wei}
\address{\noindent J. Wei -Department of Mathematics, University of British
Columbia, Vancouver, B.C., Canada, V6T 1Z2 and
Department Of Mathematics, Chinese University Of Hong Kong, Shatin,
Hong Kong }
\email{jcwei@math.ubc.ca}

\begin{abstract}
We consider the following elliptic system with fractional laplacian
$$ -(-\Delta)^su=uv^2,\ \ -(-\Delta)^sv=vu^2,\ \ u,v>0 \ \mbox{on}\ \R^n,$$
where $s\in(0,1)$ and $(-\Delta)^s$ is the $s$-Lapalcian.  We  first prove that all positive solutions must have polynomial bound. Then we use the Almgren monotonicity formula to perform a blown-down analysis to $s$-harmonic functions. Finally we use the method of moving planes to prove the uniqueness of the one dimensional profile, up to translation and scaling.

\end{abstract}



\maketitle

\date{}

\section{Introduction and main results}
\setcounter{equation}{0}

In this paper we prove the uniqueness of the positive solutions $(u, v)$, up to scaling and translations, of the following nonlocal elliptic system
\begin{equation}
\label{1p}
-(-\Delta)^s u=uv^2, -(-\Delta)^s v=  vu^2,\ u, v>0 \ \mbox{in} \ \R^1
\end{equation}
where $(-\Delta)^s$ is the $s$-laplacian with $0<s<1$.

When $s=1$, problem (\ref{1p}) arises as limiting equation in the study of phase separations in Bose-Einstein system and also in the Lotka-Volterra competition systems. More precisely, we consider the classical two-component  Lotka-Volterra competition  systems
\begin{equation}
\label{1n}
\left\{\begin{array}{l}
 -\Delta u + \beta_1 u^3 + \beta v^2 u = \lambda_1 u \ \ \text{in }\ \Omega, \\
 -\Delta v +\beta_2 v^3 + \beta u^2 v = \lambda_2 v  \ \ \text{in }\ \Omega, \\
 u>0,\quad v>0  \  \ \text{in } \ \Omega, \\
 u=0,\quad v=0 \ \  \text{on }\ \partial\Omega\,, \\
\int_\Omega u^2=N_1,\quad\int_\Omega v^2=N_2\, ,
\end{array}
\right.
\end{equation}
where $\beta_1, \beta_2, \beta >0$ and $\Omega$ is a bounded smooth domain  in $\R^n$. Solutions of (\ref{1n})
can be regarded as critical points of the energy  functional
\begin{equation}\label{5.1n}
E_\beta (u,v)=\int_\Omega\,\left(|\nabla u|^2+|\nabla
v|^2\right)+\frac{\beta_1}{2}u^4+\frac{\beta_2}{2}v^4+\frac{\beta}{2}
u^2v^2\,,\end{equation} on the space $(u,v)\in H^1_0(\Omega)\times
H^1_0(\Omega)$ with  constraints
\begin{equation}
\label{302n}
\int_\Omega u^2 dx=N_1, \int_\Omega v^2 dx=N_2.
\end{equation}

Of particular interest is the asymptotic behavior of family of bounded energy solutions $(u_\beta, v_\beta)$  in the case of {\em strong competition}, i.e., when  $\beta \to +\infty$, which  produces  spatial segregation in the limiting profiles.  After suitable scaling and blowing up process, (see Berestycki-Lin-Wei-Zhao \cite{blwz} and Noris-Tavares-Terracini-Verzini  \cite{NTTV}),
 we arrive at the following  nonlinear
elliptic system
\begin{equation}\label{maineqn}
\Delta u= u v^2\,, \quad \Delta v= v u^2\,, \quad u,
v > 0 \quad \mbox{in} \quad \R^n\,.
\end{equation}

Recently there have been intense studies on the elliptic system (\ref{maineqn}). In \cite{blwz, BTWW}  the relationship between system (\ref{maineqn}) and the
celebrated Allen-Cahn equation is emphasized. A De Giorgi's-type
and a Gibbons'-type conjecture for the solutions of (\ref{maineqn}) are formulated. Now we recall the following results for the system (\ref{maineqn}).

\medskip

\noindent
(1) When $n=1$, it has been proved that the one-dimensional profile, having linear growth, is reflectionally symmetric, i.e., there exists $x_0$ such that $ u(x-x_0)= v(x_0-x)$, and is  unique, up to translation and scaling (Berestycki-Terracini-Wang-Wei \cite{BTWW}). Furthermore this solution is nondegenerate and stable (Berestycki-Lin-Wei-Zhao \cite{blwz}).

\noindent
(2) When $n\geq 2$, all sublinear growth solutions are trivial (Noris-Tavares-Terracini-Verzini  \cite{NTTV}). Furthermore, Almgren's and Alt-Caffarelli-Friedman monotonicity formulas are derived (Noris-Tavares-Terracini-Verzini  \cite{NTTV}).

\noindent
(3) When $n=2$, the monotonic solution, i.e. $(u, v)$ satisfies
\begin{equation}
\label{mon}
\frac{\partial u}{\partial x_n}>0, \ \ \ \ \frac{\partial v}{\partial x_n} <0
\end{equation}
must be one-dimensional (Berestycki-Lin-Wei-Zhao \cite{blwz}), provided that $(u, v)$ has the following linear growth
\begin{equation}
\label{linear}
u(x)+ v(x) \leq C (1+|x|)
\end{equation}

\noindent
Same conclusion holds if we consider stable solutions (Berestycki-Terracini-Wang-Wei \cite{BTWW}). It has also been proved by Farina \cite{F} that the conditions (\ref{mon})-(\ref{linear}) can be reduced to
\begin{equation}
\label{mon1}
\frac{\partial u}{\partial x_n} >0
\end{equation}
and
\begin{equation}
\label{linear}
u(x)+ v(x) \leq C (1+|x|)^d, \ \mbox{for some positive integer} \ d.
\end{equation}
\noindent
The Gibbon's conjecture has also been solved under the polynomial growth condition (\ref{linear}) (Farina-Soave \cite{FS}).

\medskip

\noindent
(4) In $\R^2$,  for each positive integer $d$ there are solutions to (\ref{maineqn})  with polynomial growth of degree $d$ (Berestycki-Terracini-Wang-Wei \cite{BTWW}). Moreover there are solutions in $\R^2$ which are periodic in one direction and have exponential growth in another direction (\cite{SZ}).

\medskip

\noindent
(5) In two papers of the first author \cite{W1, W2}, it is proved that any solution of \eqref{maineqn} with linear growth is one dimensional.

\medskip

In \cite{TVZ}-\cite{TVZ 2}, Terracini, Verzini and Zillo initiated the study of competition-diffusion nonlinear systems involving {\em fractional Lapalcian} of the form
\begin{equation}
\label{new1}
\left\{\begin{array}{l}
(-\Delta)^s u_i= f_{i, \beta} (u_i)- \beta u_i \sum_{j \not = i} a_{ij} u_j^2, i=1,..., k\\
u_i \in H^s (\R^n)
\end{array}
\right.
\end{equation}
where $n\geq 1$, $ a_{ij}=a_{ji}$, $\beta$ is positive and large, and the fractional Lapalcian $(-\Delta)^s$  is defined as
$$ (-\Delta)^s u(x)= c_{n, s} \mbox{pv} \int_{\R^n} \frac{ u(x)-u(y)}{|x-y|^{n+2s}} dy. $$

It is well known that fractional diffusion arises when the Gaussian statistics of the classical Brownian motion is replaced by a different one, allowing for the L\'{e}vy jumps (or flights). The operator $(-\Delta)^s$ can be seen as the infinitesimal generators of L\'{e}vy stable diffusion process (Applebaum \cite{apple}). This operator arises in several areas such as physics, biology and  finance. In particular in population dynamics while the standard lapacian seems well suited to describe the diffusion of predators in presence of an abundant prey, when the prey is sparse observations suggest that fractional Lapalcians give a more accurate model (Humphries \cite{Hum}). Mathematically (\ref{new1}) is a more challenging problem because the operator is of the nonlocal nature.

In \cite{TVZ, TVZ 2, VZ}, they derived the corresponding Almgren's and Alt-Caffarelli-Friedman's monotonicity formula and proved that the bounded energy solutions have uniform H\"{o}lder regularity with small H\"{o}lder exponent $ \alpha =\alpha (N, s)$. As in the standard diffusion case,  a key result to prove is to show that there are no entire solutions to the blown-up limit system
\begin{equation}\label{equation}
-(-\Delta)^su=uv^2,\ \ -(-\Delta)^sv=vu^2,\ \ u,v>0 \ \mbox{in}\ \R^n,
\end{equation}
with small H\"{o}lder continuous exponent.

\medskip

In this paper, we study some basic qualitative behavior of solutions to (\ref{equation}), including

\medskip

\noindent
(a) are all one-dimensional solutions unique, up to translation and scaling?

\noindent
(b) do all solutions have polynomial bounds?

\medskip

We shall answer both questions  affirmatively. To state our results, we consider the Caffarelli-Silvestre extension of (\ref{equation}). Letting $a:=1-2s\in(-1,1)$,  as in \cite{C-S}, we introduce the elliptic operator
\[L_av:=\mbox{div}\left(y^a\nabla v\right),\]
for functions defined on the upper half plane $\R^2_+$. Define
\[\partial_y^a:=\lim_{y\to0^+}y^a\frac{\partial v}{\partial y}.\]

The problem \eqref{equation} is equivalent to the following extension problem
\begin{equation}\label{equation extension}
\left\{
\begin{aligned}
 &L_a u=L_a v=0, \ \mbox{in}\ \R^{n+1}_+, \\
 &\partial_y^a u=uv^2,\ \ \partial^a_y v=vu^2 \ \ \mbox{on}\ \partial\R^{n+1}_+.
                          \end{aligned} \right.
\end{equation}
Indeed, solutions of this extension problem can be seen as solutions of  \eqref{new1} in the viscosity sense.

Throughout this paper, we take the following notations. $z=(x,y)$ denotes a point in $\R^{n+1}_+$ where $x\in\R^n$ and $y\in\R_+$. In polar coordinates, $y=r\sin\theta$ where $\theta\in[0,\pi/2]$. When $n=1$, we also use the notation $z=x+iy=(r\cos\theta,r\sin\theta)$. The half ball $B_r^+(z_0)=B_r(z_0)\cap\R^{n+1}_+$, the positive part of its boundary $\partial^+B_r^+(z_0)=\partial B_r(z_0)\cap\R^{n+1}_+$ and the flat part $\partial^0B_r^+(z_0)=\partial B_r^+(z_0)\setminus\partial^+B_r^+$. Moreover, if the center of ball is the origin $0$, it will be omitted.

Note that the problem \eqref{equation extension} is invariant under the scaling $(u(z),v(z))\mapsto (\lambda^s u(\lambda z),\lambda^s v(\lambda z)$ and translations in $\R^n$ directions.

Our first main result is

\begin{thm}\label{main result}
When $n=1$ and $s\in(1/4,1)    $, the solution $(u,v)$ of \eqref{equation extension} is unique up to a scaling and translation in the $x$-direction. In particular, there exists a constant $T$ such that
\[u(x,y)=v(2T-x,y), \ \ \mbox{in}\ \R^2.\]
\end{thm}

Next we shall prove

\begin{thm}\label{main result2}
When $n\geq 1, s\in (0,1)$, the solution $(u,v)$ of \eqref{equation extension} must have at most polynomial growth: there exists $d>0$ such that
\begin{equation}
u(z)+ v(z) \leq C (1+|x|+|y|)^d.
\end{equation}
\end{thm}

\medskip

Let us put our results in broader context. The uniqueness for fractional nonlinear elliptic equations is a very challenging problem. The only results known in this direction are due to Frank-Lenzmann \cite{FL} and Frank-Lenzmann-Silvestre \cite{FLS}, in which they proved the nondegeneracy and uniqueness of radial ground states for the following fractional  nonlinear Schr\"{o}dinger equation
\begin{equation}
- (-\Delta)^s Q- Q+ Q^p=0,\  Q>0, \ Q\in H^s (\R^n).
\end{equation}
\noindent
Our proof of Theorem \ref{main result} is completely different from theirs: we make use of the method of moving planes (as in \cite{BTWW}) to prove uniqueness. To apply the method moving plane, we have to know precise asymptotics of the solutions up to high orders. This is achieved by blown-down analysis and Fourier mode expansions. (The condition that $ s>\frac{1}{4}$ seems to be technical only.)  In dealing with nonlocal equations some "trivial" facts can become quite nontrivial. For example, one of "trivial" question is whether or not one dimensional profile has linear growth. (When $s=1$ this is a trivial consequence of Hamiltonian identity. See \cite{blwz}.) To prove this for the fractional laplacian case we employ Yau's gradient estimates. A surprising result is that this also gives the polynomial bound for all solutions (Theorem \ref{main result2}). This is in sharp contrast with $s=1$ case since there are exponential growth solutions (\cite{SZ}).

\medskip

The rest of the paper is organized as follows: In Section 2 we prove Yau's estimates for
$s-$subharmonic functions from which we prove Theorem 1.2. Sections 3 and 4 contain the
Almgren's monotonicity formula and the blown-down process to $s-$harmonic functions. We
prove Theorem \ref{main result} in Sections 5-8: we first classify the blown-down limit when $n = 1$ (Section 5). Then we prove the growth bound and decay estimates (Section 6). In order to apply the
method moving planes we need to obtain refined asymptotics (Section 7). Finally we apply
the method of moving planes to prove the uniqueness result. We list some basic facts about
$s-$harmonic functions in the appendix.

\section{Gradient estimate for positive $L_a$-harmonic functions and Proof of Theorem \ref{main result2}}
\numberwithin{equation}{section}
 \setcounter{equation}{0}

In this section we prove the following Yau's type gradient estimate (cf. \cite{SY}) for positive $L_a$-harmonic functions and use it to give a polynomial bound for solutions of \eqref{equation extension}. Regarding Yau's estimates
for harmonic functions on manifolds, we refer to the book by Schoen-Yau \cite{SY}.

\begin{thm}\label{Yau's gradient estimate}
Let $u$ be a positive $L_a$-harmonic function in $\R^{n+1}_+$.
There exists a constant $C(n)$ such that
\[\frac{|\nabla u(x,y)|}{u(x,y)}\leq \frac{C(n)}{y},\ \ \ \mbox{in}\ \R^{n+1}_+.\]
\end{thm}
\begin{proof}
Let $v:=\log u$, which satisfies
\begin{equation}\label{log u}
-\Delta v=|\nabla v|^2+ay^{-1}\frac{\partial v}{\partial y}.
\end{equation}
By a direct calculation we have
\begin{equation}\label{gradient log u}
\frac{1}{2}\Delta |\nabla v|^2=|\nabla^2v|^2-\nabla|\nabla v|^2\cdot\nabla v-\frac{a}{2y}\frac{\partial}{\partial y}|\nabla v|^2
+\frac{a}{y^2}\Big|\frac{\partial v}{\partial y}\Big|^2.
\end{equation}

For any $z_0=(x_0,y_0)\in \R^{n+1}_+$, let $R=y_0/3$. Take a nonnegative function $\eta\in C_0^\infty(B_{2R}(z_0))$ and let $w:=|\nabla v|^2\eta$. Since $w$ vanishes on $\partial B_{2R}(z_0)$, it attains its maximum at an interior point, say $z_1$.

At $z_1$,
\begin{equation}\label{first order condition}
0=\nabla w=\eta\nabla|\nabla v|^2+|\nabla v|^2\nabla\eta,
\end{equation}
\begin{equation}\label{second order condition}
0\geq\Delta w=\eta\Delta|\nabla v|^2+2\nabla|\nabla v|^2\cdot\nabla\eta+|\nabla v|^2\Delta\eta.
\end{equation}
Substituting \eqref{gradient log u} and \eqref{first order condition} into \eqref{second order condition} leads to
\begin{eqnarray*}
0&\geq&2|\nabla^2v|^2\eta+2|\nabla v|^2\nabla v\cdot\nabla\eta+a y^{-1}|\nabla v|^2\frac{\partial\eta}{\partial y}\\
&&-2|\nabla v|^2\eta^{-1}|\nabla \eta|^2+|\nabla v|^2\Delta\eta.
\end{eqnarray*}
By the Cauchy inequality and \eqref{log u},
\begin{eqnarray*}
|\nabla^2v|^2&\geq&\frac{1}{n+1}\left(\Delta v\right)^2\\
&=&\frac{1}{n+1}\left(|\nabla v|^4+2ay^{-1}|\nabla v|^2\frac{\partial v}{\partial y}+\frac{a^2}{y^2}\Big|\frac{\partial v}{\partial y}\Big|^2\right).
\end{eqnarray*}
Substituting this into the above gives
\begin{eqnarray*}
0&\geq&\frac{2}{n+1}|\nabla v|^4\eta+\frac{4a}{(n+1)y}|\nabla v|^2\frac{\partial v}{\partial y}\eta+\frac{2a^2}{(n+1)y^2}\Big|\frac{\partial v}{\partial y}\Big|^2\eta\\
&&+3|\nabla v|^2\nabla v\cdot\nabla\eta+a y^{-1}|\nabla v|^2\frac{\partial\eta}{\partial y}\\
&&-2|\nabla v|^2\eta^{-1}|\nabla \eta|^2+|\nabla v|^2\Delta\eta.
\end{eqnarray*}
Now take an $\varphi\in C_0^\infty(B_{2R}(z_0))$, satisfying $0\leq\varphi\leq 1$, $\varphi\equiv 1$ in $B_R(z_0)$ and $|\nabla\varphi|^2+|\Delta\varphi|\leq 100R^{-2}$.
Choose an $m\geq 3$ and substitute $\eta=\varphi^{2m}$ into the above, which results in
\begin{eqnarray*}
|\nabla v|^4\varphi^{2m}&\leq&C(n)y^{-1}|\nabla v|^3\varphi^{2m}+C(n)y^{-2}|\nabla v|^2\varphi^{2m}\\
&&+C(n,m)\varphi^{2m-1}|\nabla v|^3|\nabla\varphi|+C(n,m)\varphi^{2m-1} y^{-1}|\nabla v|^2|\nabla\varphi|\\
&&+C(n,m)\varphi^{2m-2}|\nabla v|^2|\nabla \varphi|^2+C(n,m)\varphi^{2m-1}|\Delta\varphi||\nabla v|^2.
\end{eqnarray*}
Applying the Young inequality to the right hand side, we obtain
\begin{eqnarray*}
|\nabla v|^4\varphi^{2m}&\leq&\frac{1}{2}|\nabla v|^4\varphi^{2m}\\
&&+C(n)\left(y^{-4}\varphi^{2m}+\varphi^{2m-4}|\nabla\varphi|^4
+\varphi^{2m-2}y^{-2}|\nabla\varphi|^2+\varphi^{2m-2}|\Delta\varphi|^2\right).
\end{eqnarray*}
By our assumption on $\varphi$, and because $y^{-1}\leq 4R^{-1}$ in $B_{2R}(z_0)$, this gives
\[|\nabla v(z_0)|^4\leq |\nabla v(z_1)|^4\varphi(z_1)^{2m}\leq C(n,m)R^{-4},\]
which clearly implies the bound on $u^{-1}|\nabla u|$.
\end{proof}

A direct consequence of this gradient estimate is a Harnack inequality for positive
$L_a$-harmonic functions.
\begin{coro}\label{coro Harnack inequality}
Let $u$ be a positive $L_a$-harmonic function in $\R^{n+1}_+$. There
exists a constant $C(n)$ such that, for any $(x,y)\in\R^{n+1}_+$,
\[\sup_{B_{y/2}(x,y)}u\leq C\inf_{B_{y/2}(x,y)}u.\]
\end{coro}
Iterating this Harnack inequality using chains of balls gives an exponential growth bound on $u$. However, we can get a more precise estimate using the hyperbolic geometry.

Now we come to the proof of Theorem \ref{main result2}. In fact, we have the following polynomial bound for positive $s$-subharmonic function on $\R^n$.
\begin{thm}\label{lem polynomial growth}
Let $u\in C(\overline{\R^{n+1}})$ be a solution of the problem
\begin{equation*}
\left\{
\begin{aligned}
 &L_a u=0, \ \mbox{in}\ \R^{n+1}_+, \\
 &u>0, \ \mbox{on}\ \overline{\R^{n+1}_+}, \\
 &\partial_y^a u\geq 0, \ \ \mbox{on}\ \partial\R^{n+1}_+.
                          \end{aligned} \right.
\end{equation*}
There exists a constant $C$ depending only on the dimension $n$ and $a$ such that,
\[u(x,y)\leq Cu(0,1)\left(1+|x|^2+y^2\right)^C.\]
\end{thm}
\begin{proof}
As in \cite{M}, for any two different points $z_i=(x_i,y_i)\in\R^{n+1}_+$ and a $C^1$ curve $\gamma(t)=(\gamma_1(t),\gamma_2(t)), t\in[0,1]$ connecting them,
\begin{eqnarray*}
\log \frac{u(z_2)}{u(z_1)}&=&\int_0^1\nabla\log u(\gamma(t))\cdot\frac{d\gamma(t)}{dt}dt\\
&\leq&\int_0^1\Big|\nabla\log u(\gamma(t))\Big|\Big|\frac{d\gamma(t)}{dt}\Big|dt\\
&\leq &C\int_0^1\frac{\Big|\frac{d\gamma(t)}{dt}\Big|}{\gamma_{n+1}(t)}dt\\
&\leq &C\mbox{Length}_H(\gamma).
\end{eqnarray*}
Here $\mbox{Length}_H(\gamma)$ is the length of $\gamma$ with respect to the hyperbolic metric on $\R^{n+1}_+$,
\[ds^2:=\frac{dx^2+dy^2}{y^2}.\]
In particular, we can take $\gamma$ to be the geodesic between $z_1$ and $z_2$. This gives
\[\log \frac{u(z_2)}{u(z_1)}\leq C\mbox{dist}_H(z_1,z_2).\]
However, we know the distance function $\mbox{dist}_H$ has the form
\[\mbox{dist}_H(z_1,z_2)=\mbox{arccosh}\left(1+\frac{|x_1-x_2|^2+\left(y_1-y_2\right)^2}{2y_1y_2}\right).\]
This then implies that
\begin{equation}\label{Harnack in half-plane}
\frac{u(z_2)}{u(z_1)}\leq \left(1+\frac{|x_1-x_2|^2+\left(y_1-y_2\right)^2}{2y_1y_2}\right)^C.
\end{equation}
In particular, for any $(x,y)\in\R^{n+1}_+$,
\[u(x,y)\leq u(0,1)\left(1+\frac{|x|^2+\left(y-1\right)^2}{2y}\right)^C.\]
In particular, in $\{y\geq 1/2\}$,
\begin{equation}\label{polynomial bound 2}
u(x,y)\leq  C\left(|x|^2+y^2+1\right)^C.
\end{equation}

For every $t\in(0,1/2)$, let $P^t(x,y)$ be the Poisson kernel of the elliptic operator $\Delta+a(y+t)^{-1}\partial_y$
on $\R^{n+1}_+$. Note that when $t=0$, it is the usual Poisson kernel for the operator $L_a$.
By \cite[Section 2.4]{C-S}, modulo a constant
\[P^0(x,y)=\frac{y^{2s}}{\left(|x|^2+y^2\right)^{\frac{n+2s}{2}}}.\]
From the uniqueness of the Poisson kernel we deduce the following production rule:
for $y>t$,
\begin{equation}\label{production}
P^0(x,y+t)=\int_{\R^n}P^t(x-\xi,y)P^0(\xi,t)d\xi.
\end{equation}

Denote the Fourier transform of $P^t(x,y)$ in $x$ by $\hat{P}^t(\zeta,y)$.
$\hat{P}^0(\zeta,y)$ has the form (modulo a constant) $\Phi(y|\zeta|)$, where
\[\Phi(|\zeta|)=d_{n,s}\int_{\R^n}\left(1+|x|^2\right)^{-\frac{n+2s}{2}}e^{-\sqrt{-1}x\cdot\zeta}dx.\]
Here $d_{n,s}$ is a normalization constant.

Since $\hat{P}^0$ satisfies
\[-|\zeta|^2\hat{P}^0(\zeta,y)+\frac{\partial^2}{\partial y^2}\hat{P}^0(\zeta,y)+\frac{a}{y}\frac{\partial}{\partial y}
\hat{P}^0(\zeta,y)=0,\]
$\Phi$ satisfies
\[\Phi^{\prime\prime}(t)+at^{-1}\Phi^\prime(t)-\Phi(t)=0, \ \ \mbox{in} \ (0,+\infty).\]
By definition and the Lebesgue-Riemann lemma, $\Phi(0)=1$ and $\lim_{t\to+\infty}\Phi(t)=0$.
Then by a maximum principle argument, we know $\Phi(t)>0$ and $\Phi(t)$ is decreasing in $t$.
%Next by the comparison principle, $\Phi(t)\leq e^{-t}$.

%Let $\Psi(t):=e^t\Phi(t)$, which satisfies $0\leq\Psi\leq 1$ and
%\[\Psi^{\prime\prime}(t)-\left(2-at^{-1}\right)\Psi^\prime(t)-at^{-1}\Psi(t)=0, \ \ \mbox{in} \ (0,+\infty).\]
%Once again by the maximum principle, $\Psi$ is decreasing in $t$.

By \eqref{production},
\[\hat{P}^t(\zeta,y)=\frac{\hat{P}^0(\zeta,y+t)}{\hat{P}^0(\zeta,t)}=\frac{\Phi^0((y+t)|\zeta|)}{\Phi(t|\zeta|)}.\]
%From this we see, when $\zeta\to\infty$, $\hat{P}^t(\zeta,y)\sim \hat{P}^0(\zeta,y)$. This then implies that \[P^t(x,y)\sim \Phi^0(x,y).\]
Hence there exists a constant $C$ depending only on $n$ and $a$ so that for all $t\in[0,1/2]$,
\begin{equation}\label{2.1}
P^t(0,1-t)=\int_{\R^n}\frac{\Phi(|\zeta|)}{\Phi(t|\zeta|)}d\zeta\geq \int_{\R^n}\Phi(|\zeta|)d\zeta=P^0(0,1)\geq\frac{1}{C}.
\end{equation}

Since $P^t$ is a positive solution of
\[\Delta P^t+\frac{a}{y+t}\frac{\partial P^t}{\partial y}=0, \ \ \mbox{in}\ \R^{n+1}_+,\]
the gradient estimate Theorem \ref{Yau's gradient estimate} holds for $P^t$ with the same constant $C(n)$.
Then similar to \eqref{Harnack in half-plane}, we get
%\[\left(1+\frac{|x|^2}{2(1-t)^2}\right)^{-C}\leq\frac{P^t(x,1-t)}{P^t(0,1-t)}\leq \left(1+\frac{|x|^2}{2(1-t)^2}\right)^C.\]
\begin{equation}\label{2.2}
P^t(x,1-t)\geq P^t(0,1-t)\left(1+\frac{|x|^2}{2(1-t)^2}\right)^{-C}.
\end{equation}

By the Poisson representation,
\begin{equation}\label{Poisson 1}
u(0,1)\geq\int_{\R^n}P^t(-x,1-t)u(x,t)dx.
\end{equation}
In fact, for any $R>0$, consider the boundary value $u(x,t)\chi_{\{|x|<R\}}$, and let $w^r$ be the solution of
\begin{equation*}
\left\{
\begin{aligned}
 &L_a w^r=0, \ \mbox{in}\ B_r^+, \\
 &w^r=u(x,t)\chi_{\{|x|<R\}}, \ \ \mbox{on}\ \partial^0B_r^+,\\
&w^r=0, \ \ \mbox{on}\ \partial^0B_r^+.
                          \end{aligned} \right.
\end{equation*}
Such $w^r$ exists and is unique. By the maximum principle, as $r\to+\infty$, they are uniformly bounded and increase to
\[\int_{\{|x|<R\}}P^t(x-\zeta,y)u(\zeta,t)dx.\]
Here we have used the fact that bounded $L_a$-harmonic function in $\R^{n+1}_+$ with boundary value $u(x,t)\chi_{\{|x|<R\}}$ is unique.

By the comparison principle, for each $r>0$, $w^r\leq u$. Thus we have
\[u(0,1)\geq\int_{\{|x|<R\}}P^t(-x,1-t)u(x,t)dx.\]
Then let $R\to+\infty$ we get \eqref{Poisson 1}.

Substituting \eqref{2.1} and \eqref{2.2} into  \eqref{Poisson 1}, we see for any $t\in(0,1/2)$,
\[\int_{\R^n}\frac{u(x,t)}{\left(|x|^2+1\right)^C} dx\leq C(n,s)u(0,1).\]
%This also holds on the half plane $\{(x,y):y\geq t\}$ for any $t\in[0,1/2]$.
Integrating $t$ in $[0,1/2]$ gives
\begin{equation}\label{3.1}
\int_0^{1/2}\int_{\R^n}\frac{u(x,y)}{\left(|x|^2+1\right)^C}dxdy\leq
Cu(0,1).
\end{equation}

For any $x_0\in\R^n$ with $|x_0|>2$, by the co-area formula, we find an $r\in(1,2)$ so that
\begin{eqnarray*}
\int_{\partial^+ B_r^+(x_0)}y^au &\leq &\int_{B_{2}^+(x_0)\setminus B_1^+(x_0)}y^au(x,y)dxdy\\
&\leq& C\left(1+|x_0|^2\right)^C\int_{\left(B_{2}^+(x_0)\setminus B_1^+(x_0)\right)\cap\{0<y<1/2\}}\frac{u(x,y)}{\left(1+|x|^2\right)^C}dxdy\\
&&+\int_{\left(B_{2}^+(x_0)\setminus B_1^+(x_0)\right)\cap\{y>1/2\}}u(x,y)dxdy\\
&\leq& C\left(1+|x_0|^2\right)^C,
\end{eqnarray*}
thanks to \eqref{polynomial bound 2} and \eqref{3.1}.

After extending $u$ evenly to $B_r(x_0)$, $u$ becomes a positive $L_a$-subharmonic function, thanks to its boundary condition on $\partial^0\R^{n+1}_+$. Then Lemma \ref{lem sup bound} implies that
\[\sup_{B_{1/2}(x_0)}u\leq C(n,a)\int_{\partial B_r(x_0)}y^au \leq C\left(1+|x_0|^2\right)^C.\]

Together with \eqref{polynomial bound 2}, we get a polynomial bound for $u$ as claimed.
\end{proof}

\section{Almgren monotonicity formula}
\setcounter{equation}{0}

We first state a Pohozaev identity for the application below. %For completeness, we present the results for all $n\geq 1$.
\begin{lem}\label{lem Pohozaev}
For any $x\in\R^n$ and $r>0$,
\begin{eqnarray*}
&&\left(n-1+a\right)\int_{B_r^+(x,0)}y^a\left(|\nabla u|^2+|\nabla v|^2\right)\\
&=&r\int_{\partial^+B_r^+(x,0)}y^a\left(|\nabla u|^2+|\nabla v|^2\right)-2y^a\left(\Big|\frac{\partial u}{\partial r}\Big|^2+\Big|\frac{\partial v}{\partial r}\Big|^2\right)\\
&&+r\int_{S_r^n(x,0)}u^2v^2-n\int_{\partial^0B_r^+(x,0)}u^2v^2.
\end{eqnarray*}
\end{lem}
\begin{proof}
This can be proved by multiplying the equation \eqref{equation extension} by $z\cdot\nabla u$ (and $z\cdot\nabla v$) and integrating by parts on $B_r^+$, cf. \cite[Lemma 6.2]{C-S} and \cite[Lemma 3.10]{TVZ 2}.
\end{proof}

Let
\[E(r):=\frac{1}{r^{n-1+a}}\int_{B_r^+}y^a\left(|\nabla u|^2+|\nabla v|^2\right)+\frac{1}{r^{n-1+a}}\int_{\partial^0B_r^+}u^2v^2,\]
\[H(r):=\frac{1}{r^{n+a}}\int_{\partial^+B_r^+}y^a\left(u^2+v^2\right),\]
and $N(r):=E(r)/H(r)$.

We have the following  (cf. \cite[Theorem 3.11]{TVZ} for the $1/2$-Lapalcian case and \cite[Proposition 2.11]{TVZ 2} for general $s$-Laplacian case).

\begin{prop}[\bf Almgren monotonicity formula]\label{prop Almgren}
$N(r)$ is non-decreasing in $r>0$.
\end{prop}
\begin{proof}
Direct calculation using the equation \eqref{equation extension} shows that
\begin{eqnarray}\label{4.1}
H^\prime(r)&=&\frac{2}{r^{n+a}}\int_{\partial^+B_r^+}y^a\left(u\frac{\partial u}{\partial r}+v\frac{\partial v}{\partial r}\right) \nonumber\\
&=&\frac{2}{r^{n+1}}\int_{B_r^+}y^a\left(|\nabla u|^2+|\nabla v|^2\right)+\frac{4}{r^{n+a}}\int_{\partial^0B_r^+}u^2v^2   \\
&=&\frac{2E(r)}{r}+\frac{2}{r^{n+a}}\int_{\partial^0B_r^+}u^2v^2.\nonumber
\end{eqnarray}
Using Lemma \ref{lem Pohozaev}, we have
\begin{equation}\label{4.2}
E^\prime(r)=\frac{1}{r^{n-1+a}}\int_{\partial^+B_r^+}y^a\left(\Big|\frac{\partial u}{\partial r}\Big|^2+\Big|\frac{\partial v}{\partial r}\Big|^2\right)+\frac{1-a}{r^{n+a}}\int_{\partial^0B_r^+}u^2v^2.
\end{equation}

Combining these two, we obtain
\begin{eqnarray}\label{Almgren derivative}
\frac{1}{2}\frac{N^\prime(r)}{N(r)}&\geq&\frac{\int_{\partial^+B_r^+}y^a\left(\big|\frac{\partial u}{\partial r}\big|^2+\big|\frac{\partial v}{\partial r}\big|^2\right)}{\int_{\partial^+B_r^+}y^a\left(u\frac{\partial u}{\partial r}+v\frac{\partial v}{\partial r}\right)}
-\frac{\int_{\partial^+B_r^+}y^a\left(u\frac{\partial u}{\partial r}+v\frac{\partial v}{\partial r}\right)}{\int_{\partial^+B_r^+}y^a\left(u^2+v^2\right)}\\
&&+\frac{1-a}{N(r)}\frac{\int_{\partial^0B_r^+}u^2v^2}{\int_{\partial^+B_r^+}y^a\left(u^2+v^2\right)},\nonumber
\end{eqnarray}
which is nonnegative.
\end{proof}

Note that \eqref{4.1} also implies that
\begin{equation}\label{4.3}
\frac{d}{dr}\log H(r)=\frac{2N(r)}{r}+\frac{2\int_{\partial^0B_r^+}u^2v^2}{\int_{\partial^+B_r^+}y^a\left(u^2+v^2\right)}\geq\frac{2N(r)}{r}.
\end{equation}

Combining this with Proposition \ref{prop Almgren} we have
\begin{prop}\label{monotonocity 2}
Let $(u,v)$ be a solution of \eqref{equation extension}. If $N(R)\geq d$, then for $r>R$,
$r^{-2d}H(r)$ is nondecreasing in $r$.
\end{prop}

The following result states a doubling property of $(u,v)$.
\begin{prop}\label{doubling property}
Let $(u,v)$ be a solution of \eqref{equation} on
$B_R$. If $N(R)\leq d$, then for every $0<r_1\leq r_2\leq R$
\begin{equation}\label{eq:h_monotone}
\dfrac{H(r_2)}{H(r_1)}\leq e^{\frac{d}{1-a}}\dfrac{r_2^{2d}}{r_1^{2d}}
\end{equation}
\end{prop}
\begin{proof}
This is similar to the proof of \cite[Proposition 5.2]{BTWW}. Since for all $r\in(0,R]$, $N(r)\leq d$, by \eqref{Almgren derivative} and \eqref{4.3} we have
\begin{eqnarray*}
\frac{d}{dr}\log H(r)&\leq&\frac{2d}{r}+\frac{2\int_{\partial^0B_r^+}u^2v^2}{\int_{\partial^+B_r^+}y^a\left(u^2+v^2\right)}\\
&\leq&\frac{2d}{r}+\frac{1}{1-a}N^\prime(r).
\end{eqnarray*}
Integrating this from $r_1$ to $r_2$, since $N(r_1)\geq 0$ and $N(r_2)\leq d$, we get \eqref{eq:h_monotone}.
\end{proof}

\begin{prop}\label{equivalent condition of polynomial growth}
Let $(u,v)$ be a solution of \eqref{equation} on
$\R^{n+1}_+$. The following two conditions are equivalent:
\begin{enumerate}
\item {\bf (Polynomial growth)} There exist two positive constants $C$ and $d$ such that
\begin{equation}\label{polynomial bound 3}
u(x,y)+v(x,y)\leq C\left(1+|x|^2+y^2\right)^{\frac{d}{2}}.
\end{equation}
\item {\bf (Upper bound on $N(R)$)} There exists a positive constant $d$ such that
\[N(R)\leq d, \ \ \ \forall R>0.\]
\end{enumerate}
\end{prop}
\begin{proof}
Since the even extension of $u$ and $v$ to $\R^{n+1}$ are $L_a$-subharmonic, $(2)\Rightarrow(1)$ is a direct consequence of Proposition
\ref{doubling property} and Lemma \ref{lem sup bound}.

On the hand, if we have \eqref{polynomial bound 3}, but there exists some $R_0>0$ such that $N(R_0)\geq d+\delta$, where $\delta>0$.
By Proposition \ref{monotonocity 2}, for all $R>R_0$,
\[\sup_{\partial^+ B_R^+}\left(u^2+v^2\right)\geq H(R)\geq \frac{H(R_0)}{R_0^{2d+2\delta}}R^{2d+2\delta},\]
which clearly contradicts \eqref{polynomial bound 3}. In other words, for any $R>0$, we must have $N(R)\leq d$.
\end{proof}

\section{Blow down analysis}
\setcounter{equation}{0}

Let $(u,v)$ be a solution of \eqref{equation extension}. By Theorem \ref{lem polynomial growth} and Proposition \ref{equivalent condition of polynomial growth}, there exists a constant $d>0$ so that
\[\lim_{R\to+\infty}N(R):=d<+\infty.\]
The existence of this limit is guaranteed by the Almgren monotonicity formula ( Proposition \ref{prop Almgren}).  Note that for any $R<+\infty$, $N(R)\leq d$.

For $R\to+\infty$, define
\[u_R(z):=L(R)^{-1}u(Rz),\ \ \ \ v_R(z):=L(R)^{-1}v(Rz),\]
where $L(R)$ is chosen so that
\begin{equation}\label{normalization}
\int_{\partial^+B_1^+}u_R^2+v_R^2=1.
\end{equation}

$(u_R,v_R)$ satisfies
\begin{equation}\label{equation extension scaled}
\left\{
\begin{aligned}
 &L_a u_R=L_a v_R=0, \ \mbox{in}\ \R^{n+1}_+, \\
 &\partial^a_y u_R=\kappa_Ru_Rv_R^2,\ \ \partial_y^a v_R=\kappa_Rv_Ru_R^2 \ \ \mbox{on}\ \partial\R^{n+1}_+,
                          \end{aligned} \right.
\end{equation}
where $\kappa_R=L_R^2R^{1-a}$.

By \eqref{normalization},
\[L_R^2=R^{-1}\int_{\partial^+B_R^+}u^2+v^2.\]
By the Liouville theorem (see \cite[Propostion 3.9]{TVZ 2}), for some $\alpha>0$ small, there exists a $C_\alpha$ such that
\begin{equation}\label{coarse lower bound}
L(R)\geq C_\alpha R^\alpha.
\end{equation}
Thus $\kappa_R\to+\infty$ as $R\to+\infty$.

By Proposition \ref{doubling property}, for any $r>1$,
\[r^{-n-a}\int_{\partial^+ B_r^+}y^a\left(u_R^2+v_R^2\right)\leq r^{2d}.\]
Since $\partial^a_y u_R\geq 0$ on $\partial^+\R^{n+1}_+$, its even extension to $\R^{n+1}$ is $L_a$-subharmonic. Thus by Lemma \ref{lem sup bound} we can get
a uniform bound from the above integral bound,
\[\sup_{B_r^+}\left(u_R+v_R\right)\leq Cr^d, \  \forall r>1.\]
Then by the uniform H\"{o}lder estimate in \cite{TVZ}, for some
$\alpha\in(0,s)$, $(u_R,v_R)$ are uniformly bounded in
$C^\alpha_{loc}(\overline{\R^{n+1}_+})$.

Because $N(r;u_R,v_R)=N(Rr;u,v)\leq d$,
\begin{equation}\label{uniform H1 bound}
\int_{B_r^+}y^a\left(|\nabla u_R|^2+|\nabla v_R|^2\right)+\int_{\partial^0 B_r^+}\kappa_Ru_R^2v_R^2\leq dr^{n-1+a+2d},\ \forall r>1.
\end{equation}

After passing to a subsequence of $R$, we can assume that
$(u_R,v_R)$ converges to $(u_\infty, v_\infty)$ weakly in
$H^{1,a}_{loc}(\R^{n+1}_+)$, uniformly in
$C^\alpha_{loc}(\overline{\R^{n+1}_+})$.

Then for any $r>1$,
\begin{eqnarray*}
\int_{\partial^0 B_r^+}u_\infty^2v_\infty^2&=&\lim_{R\to+\infty}\int_{\partial^0 B_r^+}u_R^2v_R^2\\
&\leq&\lim_{R\to+\infty}\kappa_R^{-1}dr^{n-1+a+2d}=0.
\end{eqnarray*}
Thus $u_\infty v_\infty\equiv 0$ on $\partial\R^{n+1}_+$.

\begin{lem}\label{lem H1 strong convergence}
$(u_R,v_R)$ converges strongly to $(u_\infty,v_\infty)$  in
$H^{1,a}_{loc}(\R^{n+1}_+)$. $\kappa_Ru_R^2v_R^2$ converges to $0$ in $L^1_{loc}(\partial\R^2_+)$.
\end{lem}
For a proof see \cite[Lemma 4.6]{TVZ 2} (and \cite[Lemma 5.6 and Lemma 6.13]{TVZ} for the $1/2$-Lapalcian case).

\begin{coro}\label{coro limit homogeneous}
For any $r>0$,
\[N(r;u_\infty,v_\infty):=\frac{r\int_{B_r^+}y^a\left(|\nabla u_\infty|^2+|\nabla v_\infty|^2\right)}{\int_{\partial^+B_r^+}y^a\left(u_\infty^2+v_\infty^2\right)}\equiv d.\]
\end{coro}
\begin{proof}
For any fixed $r>0$, by Lemma \ref{lem H1 strong convergence},
\[\int_{B_r^+}y^a\left(|\nabla u_\infty|^2+|\nabla v_\infty|^2\right)=\lim_{R\to+\infty}\int_{B_r^+}y^a\left(|\nabla u_R|^2+|\nabla v_R|^2\right)
+\int_{\partial^0 B_r^+}\kappa_Ru_R^2v_R^2.\] By the uniform
convergence of $u_R$ and $v_R$, we also have
\[\int_{\partial^+B_r^+}y^a\left(u_\infty^2+v_\infty^2\right)=\lim_{R\to+\infty}\int_{\partial^+B_r^+}y^a\left(u_R^2+v_R^2\right).\]
Thus
\[N(r;u_\infty, v_\infty)=\lim_{R\to+\infty}N(r;u_R,v_R)=\lim_{R\to+\infty}N(Rr;u,v)=d.\qedhere\]
\end{proof}

For any $\eta\in C_0^\infty(\R^{n+1})$ nonnegative and even in $y$, multiplying the equation of $u_R$ by $\eta$ and integrating by parts, we obtain
\begin{equation}\label{Radon measure}
\int_{\partial\R^{n+1}_+}\eta\partial^a_yu_Rdx=\int_{\partial\R^{n+1}_+}\eta\kappa_Ru_Rv_R^2dx=\int_{\R^{n+1}_+}u_RL_a\eta,
\end{equation}
which is uniformly bounded as $R\to+\infty$. Hence we can assume that (up to a subsequence) $\partial^a_yu_Rdx=\kappa_Ru_Rv_R^2dx$ converges to a positive Radon measure $\mu$. On the other hand, passing to the limit in \eqref{Radon measure} gives $\mu=\partial^a_y u_\infty dx$. Here $\partial^a_y u_\infty\geq 0$ on $\partial\R^{n+1}_+$ in the weak sense, that is, $\partial^a_y u_\infty dx$ is a positive Radon measure on $\partial\R^{n+1}_+$.

\begin{lem}\label{lem limit equation 1}
The limit $(u_\infty,v_\infty)$ satisfies
\begin{equation}\label{limit equation 1}
\left\{
\begin{aligned}
 &L_a u_\infty=L_a v_\infty=0, \ \mbox{in}\ \R^{n+1}_+, \\
 &u_\infty\partial^a_y u_\infty=v_\infty\partial^a_y  v_\infty=0 \ \ \mbox{on}\ \partial\R^{n+1}_+,
                          \end{aligned} \right.
\end{equation}
\end{lem}
Here the second equation in \eqref{limit equation 1} is equivalent to the statement that the support of $\partial^a_y u_\infty dx$ belongs to $\{u_\infty=0\}$.
\begin{proof}
The first equation can be directly obtained by passing to the limit in $L_a u_R=L_a v_R=0$ and using the uniform convergence of $(u_R, v_R)$.

To prove the second one, take an arbitrary point $z_0=(x_0,0)\in \{u_\infty>0\}$. Since $u_\infty$ is continuous, we can find an $r_0>0$ and $\delta_0>0$ such that $u_\infty\geq 2\delta_0$ in $B_{r_0}^+(z_0)$.
By the segregated condition, $v_\infty(z_0)=0$. Thus by decreasing $r_0$ if necessary, we can assume that
\[v_\infty\leq\delta_0\ \ \mbox{in}\ \overline{B_{r_0}^+(z_0)}.\]
Then by the uniform convergence of $u_R$ and $v_R$, for all $R$ large,
\[
u_R\geq \delta_0, \ \ v_R\leq 2\delta_0\ \ \mbox{in}\ \overline{B_{r_0}^+(z_0)}.
\]
Thus
\[\partial^a_y v_R\geq \kappa_R\delta_0^2v_R\ \ \mbox{on}\ \partial B_{r_0}^+(z_0).\]
By applying Lemma \ref{lem decay}, we obtain
\[\sup_{\partial^0 B_{r_0/2}^+(z_0)}v_R\leq C(r_0,\delta_0)\kappa_R^{-1}.\]
Then because $\partial^a_y u_R=\kappa_Ru_Rv_R^2$, it is uniformly bounded in $C^\beta(\partial_0 B_{r_0}^+(z_0))$ for some $\beta>0$.

Let $w_R=y^a\frac{\partial u_R}{\partial y}$. It can be directly checked that $w_R$ satisfies (see \cite[Section 2.3]{C-S})
\[\mbox{div}\left(y^{-a}\nabla w_R\right)=0.\]

By \eqref{uniform H1 bound},
\[\int_{B_{r_0/2}^+(z_0)}y^{-a}w_R^2\leq\int_{B_{r_0/2}^+(z_0)}y^{a}|\nabla u_R|^2\]
are uniformly bounded. Then by the boundary H\"{o}lder estimate (\cite{RS}), $w_R$ are uniformly bounded in $C^{\beta}(\overline{B_{r_0/2}^+(z_0)})$. Because $w_R\geq 0$ on $\partial^0 B_{r_0}^+(z_0)$ and $w_Ru_R\to0$ in $L^1(\partial^0 B_{r_0}^+(z_0))$, by letting $R\to +\infty$ and using the uniform H\"{o}lder continuity of $u_R$ and $w_R$, we get
\[\partial^a_y u_\infty=0 \ \ \ \mbox{on}\ \partial^0B_{r_0/2}^+(z_0).\]

In the blow down procedure, we have shown that $u_\infty\in C^\alpha_{loc}(\overline{\R^{n+1}_+})$ for some $\alpha>0$. Hence $u_\infty$ is continuous on
$\partial \R^{n+1}_+$. The above argument also shows that $y^a\partial_y u_\infty$ is continuous up to
$\{u_\infty>0\}\cap\partial \R^{n+1}_+$ and $\partial^a_y u_\infty=0$ on $\{u_\infty>0\}\cap\partial \R^{n+1}_+$. This completes the proof.
\end{proof}

Integrating by parts using \eqref{limit equation 1}, we get
\begin{equation}\label{2}
\int_{\partial^+B_r^+}y^au_\infty\frac{\partial u_\infty}{\partial r}=\int_{B_r^+}y^a|\nabla u_\infty|^2, \ \
\int_{\partial^+B_r^+}y^av_\infty\frac{\partial v_\infty}{\partial r}=\int_{B_r^+}y^a|\nabla v_\infty|^2,
\end{equation}
for any ball $B_r^+$.

Let
\[E_\infty(r):=r^{1-n-a}\int_{B_r^+}y^a\left(|\nabla u_\infty|^2+|\nabla v_\infty|^2\right),\]
\[H_\infty(r):=r^{-n-a}\int_{\partial^+B_r^+}y^a\left(u_\infty^2+v_\infty^2\right),\]
and $N_\infty(r):=E_\infty(r)/H_\infty(r)$.

By \eqref{2} and calculating as in \eqref{4.1}, we still have
\begin{equation}\label{3}
\frac{d}{dr}\log H_\infty(r)=\frac{2N_\infty(r)}{r}.
\end{equation}
Since $N_\infty(r)\equiv d$, integrating this and by noting the normalization condition \eqref{normalization}, which passes to the limit, gives
\begin{equation}\label{4}
H_\infty(r)\equiv r^{2d}.
\end{equation}

The following lemma is essentially \cite[Propostion 2.11]{TVZ 2}.
\begin{lem}
For any $r\in(0,+\infty)$, $H_\infty(r)>0$ and $E_\infty(r)>0$. Moreover,
\begin{equation}\label{Almgren derivative limit}
\frac{1}{2}\frac{N_\infty^\prime(r)}{N_\infty(r)}\geq
\frac{\int_{\partial^+B_r^+}y^a\left(\big|\frac{\partial u_\infty}{\partial r}\big|^2+\big|\frac{\partial v_\infty}{\partial r}\big|^2\right)}{\int_{\partial^+B_r^+}y^a\left(u_\infty\frac{\partial u_\infty}{\partial r}+v_\infty\frac{\partial v_\infty}{\partial r}\right)}
-\frac{\int_{\partial^+B_r^+}y^a\left(u_\infty\frac{\partial u_\infty}{\partial r}+v_\infty\frac{\partial v_\infty}{\partial r}\right)}{\int_{\partial^+B_r^+}y^a\left(u_\infty^2+v_\infty^2\right)}\geq0,
\end{equation}
in the distributional sense.
\end{lem}
\begin{proof}
The Pohozaev identity for $(u_R,v_R)$ reads as
\begin{eqnarray*}
&&\left(n-1+a\right)\int_{B_r^+}y^a\left(|\nabla u_R|^2+|\nabla v_R|^2\right)\\
&=&r\int_{\partial^+B_r^+}y^a\left(|\nabla u_R|^2+|\nabla v_R|^2\right)-2y^a\left(\Big|\frac{\partial u_R}{\partial r}\Big|^2+\Big|\frac{\partial v_R}{\partial r}\Big|^2\right)\\
&&+r\int_{S_r^n}\kappa_Ru_R^2v_R^2-n\int_{\partial^0B_r^+}\kappa_Ru_R^2v_R^2.
\end{eqnarray*}

By Lemma \ref{lem H1 strong convergence}, for all but countable $r\in(0,+\infty)$, we can pass to the limit in the above identity, which gives
\begin{eqnarray}\label{Pohozaev limit}
&&\left(n-1+a\right)\int_{B_r^+}y^a\left(|\nabla u_\infty|^2+|\nabla v_\infty|^2\right)\\
&=&r\int_{\partial^+B_r^+}y^a\left(|\nabla u_\infty|^2+|\nabla v_\infty|^2\right)-2y^a\left(\Big|\frac{\partial u_\infty}{\partial r}\Big|^2+\Big|\frac{\partial v_\infty}{\partial r}\Big|^2\right). \nonumber
\end{eqnarray}
The following calculation is similar to the proof of Proposition \ref{prop Almgren}.
\end{proof}

\begin{lem}\label{lem limit homogeneous}
For any $\lambda>0$,
\[u_\infty(\lambda z)=\lambda^du_\infty(z),\ \ \ v_\infty(\lambda z)=\lambda^dv_\infty(z).\]
\end{lem}
\begin{proof}
By Corollary \ref{coro limit homogeneous}, $N_\infty(r)\equiv d$. Then by the previous lemma, for a.a. $r>0$,
\[\frac{\int_{\partial^+B_r^+}y^a
\left(\big|\frac{\partial u_\infty}{\partial
r}\big|^2+\big|\frac{\partial v_\infty}{\partial r}\big|^2\right)}
{\int_{\partial^+B_r^+}y^a \left(u_\infty\frac{\partial
u_\infty}{\partial r}+v_\infty\frac{\partial v_\infty}{\partial
r}\right)}
-\frac{\int_{\partial^+B_r^+}y^a\left(u_\infty\frac{\partial
u_\infty}{\partial r}+v_\infty\frac{\partial v_\infty}{\partial
r}\right)}{\int_{\partial^+B_r^+}y^a\left(u_\infty^2+v_\infty^2\right)}=0.\]
By the characterization of the equality case in the Cauchy
inequality, there exists a $\lambda(r)>0$, such that
\[\frac{\partial u_\infty}{\partial r}=\lambda(r)u_\infty,\ \ \
\frac{\partial v_\infty}{\partial r}=\lambda(r)v_\infty\ \ \
\mbox{on}\ \partial^+B_r^+.\] Integrating this in $r$, we then get
two functions $g(r)$ defined on $(0,+\infty)$ and
$(\varphi(\theta),\psi(\theta))$ defined on $\partial^+B_1^+$, such
that
\[u_\infty(r,\theta)=g(r)\varphi(\theta),\ \ \ \ v_\infty(r,\theta)=g(r)\psi(\theta).\]
By \eqref{4}, we must have $g(r)\equiv r^d$.
\end{proof}

\section{Classification of the blow down limit in dimension $2$}
\setcounter{equation}{0}

Now assume $n=1$. In the previous section we proved that the blow down limit
\[u_\infty(r,\theta)=r^d\phi(\theta),\ \ \ v_\infty(r,\theta)=r^d\psi(\theta),\]
where the two functions $\phi$ and $\psi$ are defined on $[0,\pi]$.

By denoting
\[L_\theta^a\varphi=\varphi_{\theta\theta}+a\cot\theta\varphi_\theta,\]
the equation for $(\varphi,\psi)$ reads as
\begin{equation}\label{equation on sphere}
\left\{
\begin{aligned}
 &L_\theta^a\varphi+d(d+a)\varphi=L_\theta^a\psi+d(d+a)\psi=0, \ \mbox{in}\ (0,\pi), \\
 &\varphi\partial^a_\theta\varphi=\psi\partial^a_\theta\psi=0, \ \mbox{at}\ \{0,\pi\},\\
 &\varphi(0)\psi(0)=\varphi(\pi)\psi(\pi)=0.
                          \end{aligned} \right.
\end{equation}
Here $\partial^a_\theta\phi(0)=\lim_{\theta\to 0}\left(\sin\theta\right)^a\varphi_\theta(\theta)$, and we have a similar one at $\pi$.

First we note that if $\varphi (0) \not =0, \varphi (\pi) \not = 0$, then $\psi(0)=\psi(\pi)=0$. We claim that $\psi \equiv 0$. In fact,  since $v_\infty$ is homogeneous of degree $2s-1$ and $L_a$-harmonic in $\R^2_+$, by \cite[Proposition 3.1]{TVZ 2}, $v_\infty\equiv 0$ in $\R^2_+$. Thus for nontrivial solutions, we must have either $ \varphi (0)\not =0, \varphi (\pi)=0$ or $ \varphi (0)=0, \varphi (\pi ) \not =0$.

There are exact nontrivial solutions to (\ref{equation on sphere}):
\begin{itemize}
\item When $d=s$, $(\varphi,\psi)=((\cos\frac{\theta}{2})^{2s},(\sin\frac{\theta}{2})^{2s})$.
\item When $s>1/2$, there is a second solution $(\varphi,\psi)=(1,0)$. This corresponds to $d=2s-1=-a$ and $(u_\infty,v_\infty)=((x^2+y^2)^{2s-1},0)$.
\end{itemize}
By Corollary \ref{coro limit homogeneous}, either $\lim_{R\to+\infty}N(R)=s$ or $\lim_{R\to+\infty}N(R)=2s-1$ (when $s>1/2$).
Thus the blow down limit can only be one of the above two, independent of subsequences of $R\to+\infty$.

\subsection{Self-segregation} Here we exclude the possibility that the blow down limit $(\varphi,\psi)=(1,0)$ when $s>1/2$.

Assume the blow down limit $(\varphi,\psi)=(1,0)$.
First we claim that
\begin{lem}\label{lem 6.1}
There exists a constant $c>0$ such that
\[u\geq c \ \ \mbox{on}\ \partial\R^2_+.\]
\end{lem}
\begin{proof}
Assume that we have a sequence $R_i$ such that $u(R_i,0)\to0$. Then necessarily $R_i\to\infty$.
Let $(u_{R_i},v_{R_i})$ be the blow down sequence defined as before. Then $(u_{R_i},v_{R_i})$ converges to $(r^{2s-1},0)$ (modulo a normalization constant) in $C_{loc}(\overline{\R^2_+})$. However, by our assumption, because $L(R_i)\to+\infty$ (see \eqref{coarse lower bound}),
\[u_{R_i}(1,0)=\frac{u(R_i,0)}{L(R_i)}\to0,\]
which is a contradiction.
\end{proof}

By the bound on $N(R)$ and Proposition \ref{doubling property}, there exists a constant $C$ such that
\[\int_{B_r^+}y^a\left(u^2+v^2\right)\leq Cr^{2+2(2s-1)}\ \ \forall\ r>1.\]

For each $r>1$, let
\[\tilde{u}(z)=u(rz),\ \ \ \tilde{v}(z)=v(rz).\]
Then $\tilde{v}$ satisfies
\begin{equation*}
\left\{
\begin{aligned}
 &L_a\tilde{v}=0, \ \mbox{in}\ B_1^+, \\
 &\partial^a_y\tilde{v}=r^{1-a}\tilde{u}^2\tilde{v}\geq cr^{1-a}\tilde{v}, \ \ \mbox{on}\ \partial^0 B_1^+.
                          \end{aligned} \right.
\end{equation*}
Here we have used the previous lemma which says $\tilde{u}\geq c$ on $\partial B_1^+$.

Applying Lemma \ref{lem decay}, we obtain
\[\sup_{\partial^0B_{1/2}^+}\tilde{v}\leq Cr^{-1}.\]
Letting $r\to+\infty$, we see $v\equiv 0$ on $\partial\R^2_+$.

Now since the growth bound of $v$ is controlled by $r^{2s-1}$, applying \cite[Proposition 3.1]{TVZ 2}, we get $v\equiv 0$ in $\R^2_+$.

The equation for $u$ becomes
\begin{equation*}
\left\{
\begin{aligned}
 &L_au=0, \ \mbox{in}\ \R^2_+, \\
 &\partial^a_y u=0, \ \ \mbox{on}\ \partial \R^2_+.
                          \end{aligned} \right.
\end{equation*}
Because the growth bound of $u$ is controlled by $r^{2s-1}$, applying \cite[Corollary 3.3]{TVZ 2}, $u$ is a constant. This is a contradiction with the condition on $N(R)$.

\subsection{}
We have proved that the blow down limit must be
\[u_\infty=a_+r^s(\cos\frac{\theta}{2})^{2s},\ \ \ v_\infty=a_-r^s(\sin\frac{\theta}{2})^{2s},\]
for two suitable positive constants $a_+$ and $a_-$.

Here we note that for any $R\to+\infty$, the blow down sequence
could also be
\[u_\infty=b_+r^s(\sin\frac{\theta}{2})^{2s},\ \ \ v_\infty=b_-r^s(\cos\frac{\theta}{2})^{2s}.\]
However by continuity, only one of them is possible and the blow
down limit must be unique (the constant $a_+$ and $a_-$ will be shown to be independent of the choice of subsequences $R_i\to+\infty$ in the next section). For example, if both these two arise as the blow
down limit (from different subsequence of $R\to+\infty$), then we
can find a sequence of $R_i\to+\infty$ satisfying
$u(R_i,0)=v(R_i,0)$. Using these $R_i$ to define the blow down
sequence, we get a blow down limit $(u_\infty,v_\infty)$ satisfying
$u_\infty(1,0)=v_\infty(1,0)$. This is a contradiction with the two
forms given above.

\begin{lem}\label{lem limit equation 2}
$a_+=a_-$.
\end{lem}
This can be proved by the Pohozaev identity for $(u_\infty,v_\infty)$, \eqref{Pohozaev limit}, where we replace the ball
$B_r^+$ by $B_r^+(t,0)$ and let $t$ vary.

We have proved that the blow down limit $(u_\infty,v_\infty)$
satisfies
\[\int_{\partial^+B_1^+}y^au_\infty^2=\int_{\partial^+B_1^+}y^av_\infty^2.\]
By the convergence of the blow down limit, we get a constant $C$ so
that for all $R\geq 1$,
\begin{equation}\label{balance between u and v}
\frac{1}{C}\leq\frac{\int_{\partial^+B_R^+}y^au^2}{\int_{\partial^+B_R^+}y^av^2}\leq
C.
\end{equation}

\section{Growth bound}
\setcounter{equation}{0}

\begin{prop}[{\bf Upper bound}]\label{upper bound}
There exists a constant $C$ so that
\[u(z)+v(z)\leq C\left(1+|z|\right)^s.\]
\end{prop}
\begin{proof}
Because for any $r$, $N(r)\leq s$. Proposition \ref{doubling property} implies that
\[H(r)\leq H(1)r^{2s},\ \ \ \forall r>1.\]
Then because the even extension of $u$ to $\R^2$ is $L_a$-subharmonic, by Lemma \ref{lem sup bound} we get
\[\sup_{B_{r/2}}u\leq CH(r)^{1/2}\leq CH(1)^{1/2}r^s.  \qedhere\]
\end{proof}

Because for any $R>0$, $N(R)\leq s$, the bound on $H(r)$ also gives
\begin{coro}\label{bound on energy}
For any $R>1$,
\[\int_{B_R^+}y^a\left(|\nabla u|^2+|\nabla v|^2\right)+\int_{\partial^0B_R^+}u^2v^2\leq CR.\]
\end{coro}

Next we give a lower bound for the growth of $u$ and $v$.
\begin{prop}[{\bf Lower bound}]\label{lower bound}
There exists a constant $c$ such that
\begin{equation}\label{sharp lower bound}
\int_{\partial^+B_r^+}y^au^2\geq cr^2,\ \ \ \int_{\partial^+B_r^+}y^av^2\geq cr^2, \ \ \ \forall \ r>1.
\end{equation}
\end{prop}
We first present two lemmas needed in the proof of this proposition.
\begin{lem}\label{lem 7.1}
For any $K>0$, there exists an $R(K)$ such that $\{Kx>y>0\}\cap
B_{R(K)}^c\subset\{u>v\}$ and $\{-Kx>y>0\}\cap
B_{R(K)}^c\subset\{u<v\}$.
\end{lem}
\begin{proof}
This is because, there exists a $\delta(K)>0$ so that for any $R\geq R(K)$,
\[\sup_{\overline{B_1^+}}\big|u^R-ar^s\left(\cos\frac{\theta}{2}\right)^{2s}\big|
+\big|v^R-ar^s\left(\sin\frac{\theta}{2}\right)^{2s}\big|\leq\delta(K),\]
 and
\[ar^s\left(\cos\frac{\theta}{2}\right)^{2s}\geq ar^s\left(\sin\frac{\theta}{2}\right)^{2s}+\delta(K), \ \ \ \mbox{in}\ \{Kx>y>0\}\cap\left(\overline{B_1^+}\setminus B_{1/2}^+\right).\]
These two imply that $u^R>v^R$ in $\overline{B_1^+}\setminus B_{1/2}^+$. By noting that this holds for any $R\geq R(K)$, we complete the proof.
\end{proof}
\begin{lem}\label{lem 7.2}
As $x\to+\infty$, $u(x,0)\to+\infty$ and $v(x,0)\to0$. As $x\to-\infty$, $v(x,0)\to+\infty$ and $u(x,0)\to0$.
\end{lem}
\begin{proof}
For any $\lambda>0$ large, let
\[u^\lambda(x,y):=\lambda^{-s}u(\lambda x,\lambda y), \ \ \
v^\lambda(x,y):=\lambda^{-s}v(\lambda x,\lambda y).\]

By the previous lemma and Proposition \ref{upper bound},
\[u^\lambda\geq v^\lambda, \ \ \ v^\lambda\leq C\ \ \ \mbox{in}\ B_{1/2}^+(1,0).\]

$v^\lambda$ satisfies
\begin{equation*}
\left\{
\begin{aligned}
 &L_a v^\lambda=0, \ \mbox{in}\ B_{1/2}^+(1,0), \\
 &\partial^a_yv^\lambda=\lambda^{4s}\left(u^\lambda\right)^2v^\lambda\geq \lambda^{4s}\left(v^\lambda\right)^3, \ \mbox{on}\ \partial B_{1/2}^+(1,0).
                          \end{aligned} \right.
\end{equation*}
Then $(v^\lambda-\lambda^{-\frac{4s}{3}})^+$ satisfies
\begin{equation*}
\left\{
\begin{aligned}
 &L_a (v^\lambda-\lambda^{-\frac{4s}{3}})^+\geq 0, \ \mbox{in}\ B_{1/2}^+(1,0), \\
 &\partial^a_y(v^\lambda-\lambda^{-\frac{4s}{3}})^+\geq \lambda^{\frac{4s}{3}}(v^\lambda-\lambda^{-\frac{4s}{3}})^+, \ \mbox{on}\ \partial B_{1/2}^+(1,0).
                          \end{aligned} \right.
\end{equation*}
By Lemma \ref{lem decay} we get
\[\sup_{B_{1/4}^+(1,0)}v^\lambda\leq \sup_{B_{1/4}^+(1,0)}(v^\lambda-\lambda^{-\frac{4s}{3}})^++\lambda^{-\frac{4s}{3}}\leq C\lambda^{-\frac{4s}{3}}.\]
Rescaling back we get $v(\lambda,0)\leq C\lambda^{-s/3}$ for all $\lambda$ large.

Next assume that there exists $\lambda_i\to+\infty$, $u(\lambda_i,0)\leq M$ for some $M>0$. Then define the blow down sequence $(u^{\lambda_i},v^{\lambda_i})$ as before, as in the proof of Lemma \ref{lem 6.1} we can get a contradiction. Indeed, the blow down analysis gives $u^{\lambda_i}(1,0)\to a$ for some constant $a>0$, but \eqref{coarse lower bound} implies that \[u^{\lambda_i}(1,0)\leq CM\lambda_i^{-\alpha}\to0.\]
This is a contradiction.
\end{proof}

Now we can prove Proposition \ref{lower bound}.
\begin{proof}[Proof of Proposition \ref{lower bound}]
By the previous lemma, there exists a constant $M^\ast$ such that $(u-M^\ast)_+$ and $(v-M^\ast)_+$ has disjoint supports on $\partial\R^2_+$. Let $w_1=(u-M)_+$ and $w_2=(v-M)_+$. Both functions are nonnegative, continuous and $L_a$-subharmonic. By assuming $M^\ast>\max\{u(0,0),v(0,0)\}$, $w_1(0,0)=w_2(0,0)=0$. Moreover, they satisfy
\begin{equation*}
\left\{
\begin{aligned}
 &w_iL_aw_i=0, \ \mbox{in}\ \R^2_+, \\
 &\partial^a_y w_i\geq 0 \ \ \mbox{on}\ \partial\R^2_+.
                          \end{aligned} \right.
\end{equation*}
This then implies that for any nonnegative $\phi\in C_0^\infty(\R^2)$,
\begin{equation}\label{9.2}
\int_{\R^2_+}y^a\nabla w_i\cdot\nabla\left(w_i\phi\right)=-\int_{\partial\R^2_+}w_i\partial^a_yw_i\phi\leq 0.
\end{equation}
Then by \cite[Proposition 2.9]{TVZ 2} (note that here the dimension $n=1$ and hence in that proposition the exponent $\nu^{ACF}=s$),
\[J(r):=r^{-4s}\left(\int_{B_r^+}y^a\frac{|\nabla w_1|^2}{|z|^a}\right)\left(\int_{B_r^+}y^a\frac{|\nabla w_2|^2}{|z|^a}\right)\]
is non-decreasing in $r>0$. This then implies the existence of a constant $c$ so that
\begin{equation}\label{9.1}
\left(\int_{B_r^+}y^a\frac{|\nabla w_1|^2}{|z|^a}\right)\left(\int_{B_r^+}y^a\frac{|\nabla w_2|^2}{|z|^a}\right)\geq cr^{4s},
\ \ \forall \ r>R^\ast.
\end{equation}
Here we choose $R^\ast$ large so that $w_1$ and $w_2$ are not constant in $B_{R^\ast}^+$, which implies
\[\int_{B_{R^\ast}^+}y^a\frac{|\nabla w_1|^2}{|z|^a}\geq c,\ \  \
\int_{B_{R^\ast}^+}y^a\frac{|\nabla w_2|^2}{|z|^a}\geq c,\]
where $c>0$ is a constant depending on the solution $(u,v)$ and $R^\ast$.

Take an $\eta\in C_0^\infty(B_2)$ such that $\eta\equiv 1$ in $B_1$. For any $r>1$, let $\eta^r(z)=\eta(r^{-1}z)$. Substituting $\phi=(\eta^r)^2$ into \eqref{9.2} and integrating by parts gives (cf. the derivation of \cite[Eq. (2.3)]{TVZ 2})
\begin{equation}\label{6.1}
\int_{B_r^+}y^a\frac{|\nabla w_i|^2}{|z|^a}\leq Cr^{-2-a}\int_{B_{2r}^+\setminus B_r^+}y^aw_i^2.
\end{equation}
Substituting this into \eqref{9.1} leads to
\[\int_{B_{2r}^+}y^a\left(u^2+v^2\right)\geq \int_{B_{2r}^+}y^a\left(w_1^2+w_2^2\right)\geq cr^{2+a+2s}.\]

Because $u^2$ and $v^2$ are $L_a$-subharmonic, by the mean value
inequality, this can be transformed to
\[\int_{\partial^+ B_r^+}y^a\left(u^2+v^2\right)\geq cr^{2},\ \ \forall r>2.\]
Then by noticing \eqref{balance between u and v}, we finish the
proof.
\end{proof}

\begin{rmk}\label{rmk 1}
With Proposition \ref{upper bound} and Proposition \ref{lower bound} in hand, in the blow down analysis we can choose
\[u^R(z):=R^{-s}u(Rz),\ \  v^R(z):=R^{-s}v(Rz).\]
By the blow down analysis, for any $R_i\to+\infty$, there exists a subsequence of $R_i$ (still denoted by $R_i$) such that
\[u^{R_i}\to br^s\left(\cos\frac{\theta}{2}\right)^{2s},\ \
v^{R_i}\to br^s\left(\sin\frac{\theta}{2}\right)^{2s},\]
in $C(\overline{B_1^+})\cap H^{1,a}(B_1^+)$, for some constant $b>0$.

We claim that $b$ is independent of the sequence $R_i$, thus the blow down limit is unique.
By \eqref{6.1} and Proposition \ref{upper bound},
\begin{equation}\label{limit of J(R)}
\lim_{R\to+\infty}J(R)<+\infty,
\end{equation}
where the limit exists because $J(R)$ is non-decreasing.

For each $R$, let $w_1^R=(u_R-M^\ast R^{-s})_+=R^{-s}w_1(Rz)$ and $w_2^R=(v_R-M^\ast R^{-s})_+=R^{-s}w_2(Rz)$. Then a rescaling gives
\[J(R_i)=\left(\int_{B_1^+}y^a\frac{|\nabla w_1^{R_i}|^2}{|z|^a}\right)\left(\int_{B_1^+}y^a\frac{|\nabla w_2^{R_i}|^2}{|z|^a}\right).\]
For any $\delta>0$ small, by \eqref{6.1},
\[\lim_{i\to+\infty}\int_{B_\delta^+}y^a\frac{|\nabla w_1^{R_i}|^2}{|z|^a}\leq C\lim_{i\to+\infty}\left(\sup_{B_{2\delta}^+}w_1^{R_i}\right)^2=O(\delta^{2s}),\]
because $w_1^{R_i}$ converges uniformly to $br^s\left(\cos\frac{\theta}{2}\right)^{2s}$.
Using this estimate and the strong convergence of $w_1^{R_i}$ in $H^{1,a}(B_1^+)$,
 we obtain
\begin{eqnarray*}
\lim_{i\to+\infty}\int_{B_1^+}y^a\frac{|\nabla w_1^{R_i}|^2}{|z|^a}&=&\lim_{i\to+\infty}\int_{B_1^+\setminus B_\delta^+}y^a\frac{|\nabla w_1^{R_i}|^2}{|z|^a}+\lim_{i\to+\infty}\int_{B_\delta^+}y^a\frac{|\nabla w_1^{R_i}|^2}{|z|^a}\\
&=&b^2\int_{B_1^+\setminus B_\delta^+}y^a\frac{|\nabla r^s\left(\cos\frac{\theta}{2}\right)^{2s}|^2}{|z|^a}+O(\delta^{2s}).
\end{eqnarray*}
After applying \eqref{6.1} to $br^s\left(\cos\frac{\theta}{2}\right)^{2s}$ and letting $\delta\to0$, this gives
\[\lim_{i\to+\infty}\int_{B_1^+}y^a\frac{|\nabla w_1^{R_i}|^2}{|z|^a}=C(s)b^2,\]
where $C(s)$ is a constant depending only on $s$.

Substituting this into \eqref{limit of J(R)} we get
\[C(s)^2b^4=\lim_{R\to+\infty}J(R).\]
Thus $b$ does not depend on the choice of subsequence $R_i$.

After a scaling $(u(z),v(z))\mapsto (\lambda^s u(\lambda z),\lambda^s v(\lambda z)$ with a suitable $\lambda$, which leaves the equation \eqref{equation extension} invariant, we can assume $b=1$. That is, as $R\to+\infty$,
\[R^{-s}u(Rz)\to r^s\left(\cos\frac{\theta}{2}\right)^{2s},\ \
R^{-s}v(Rz)\to r^s\left(\sin\frac{\theta}{2}\right)^{2s}.\]
\end{rmk}

By \eqref{sharp lower bound} and Proposition \ref{upper bound},
there exist two constants $c_1,c_2>0$, such that
\[\{u>c_1r^s\}\cap \partial^+B_r^+\cap\{y\geq c_2|x|\}\neq\emptyset.\]
 Since $u$ is positive $L_a$-harmonic in $\R^2_+$,  by applying the
Harnack inequality to a chain of balls (with the number of balls
depending only on $\varepsilon$), for any $\varepsilon>0$, there
exists a constant $c(\varepsilon)$ such that
\begin{equation}\label{sharp lower bound 2}
u(z)\geq c(\varepsilon)|z|^s, \ \ v(z)\geq c(\varepsilon)|z|^s\ \ \
\mbox{in}\ \{y\geq \varepsilon|x|\}.
\end{equation}

\begin{lem}
For any $\varepsilon,\delta>0$, there exists a constant
$R(\varepsilon,\delta)$ such that
\begin{equation}\label{5}
v(z)\leq \delta|z|^s, \ \ \mbox{in} \ \{(x,y): x\geq
R(\varepsilon,\delta), 0\leq y\leq \varepsilon
(x-R(\varepsilon,\delta)) \}.
\end{equation}
\end{lem}
\begin{proof}
By Proposition \ref{upper bound} and Proposition \ref{lower bound},
in the definition of blow down sequence we can take
\[u_R(z)=R^{-s}u(Rz),\ \ \ v_R(z)=R^{-s}v(Rz).\]
As $R\to+\infty$, $(u_R,v_R)$ converges to
$(r^s(\cos\frac{\theta}{2})^{2s},r^s(\sin\frac{\theta}{2})^{2s})$
uniformly in $B_1^+$. Thus we can choose an $\varepsilon$ depending
only on $\delta$ so that for all $R$ large, $v_R\leq\delta$ in
$B_1^+\cap\{0\leq y\leq\varepsilon x^+\}$.
\end{proof}

\begin{lem}
For any $\varepsilon>0$, there exists a constant $c(\varepsilon)$
such that
\begin{equation}\label{sharp lower bound 3}
u(z)\geq c(\varepsilon)|z|^s, \ \ \ \mbox{in}\ \{y\geq \varepsilon
x_-\}.
\end{equation}
\end{lem}
\begin{proof}
In view of \eqref{sharp lower bound 2}£¬ we only need to give a
lower bound in the domain $\mathcal{C}:=\{(x,y): x\geq
R_0, 0\leq y\leq \varepsilon (x-R_0) \}$, where $R_0$ is large but fixed.

$u-v$ is $L_a$-harmonic in $\mathcal{C}$, satisfying the following
boundary conditions (thanks to Lemma \ref{lem 7.1})
\begin{equation*}
\left\{
\begin{aligned}
 &u-v\geq c(\varepsilon)|z|^s, \ \mbox{on}\ \{ y=\varepsilon (x-R_0)\}\cap\mathcal{C}, \\
 &\partial^a_y(u-v)=uv^2-vu^2\leq 0, \ \mbox{on}\ \{ y=0\}\cap\partial\mathcal{C}.
                          \end{aligned} \right.
\end{equation*}
We claim that $u-v\geq c(\varepsilon, R_0)r^s$ %by the comparison principle,
in $\mathcal{C}$.

First, let $\psi(\theta)$ be the solution of
\begin{equation*}
\left\{
\begin{aligned}
 &-L^a_\theta\psi=d(d+a)\psi, \ \mbox{in}\ \{-\varepsilon<\theta<\varepsilon\}, \\
 &\psi>0, \ \mbox{in}\ \{-\varepsilon<\theta<\varepsilon\}, \\
 &\psi(-\varepsilon)=\psi(\varepsilon)=0.
                          \end{aligned} \right.
\end{equation*}
Here $d$ is determined by
\[d(d+a)=\min\frac{\int_{-\varepsilon}^\varepsilon\psi^\prime(\theta)^2\left(\sin\theta\right)^ad\theta}
{\int_{-\varepsilon}^\varepsilon\psi(\theta)^2\left(\sin\theta\right)^ad\theta},\]
in the class of functions satisfying $\psi(-\varepsilon)=\psi(\varepsilon)=0$.

This minima can be bounded from below by
\[c\min_{\eta\in C_0^\infty((-\varepsilon,\varepsilon))}\frac{\int_{-\varepsilon}^\varepsilon x^a\eta^\prime(x)^2dx}
{\int_{-\varepsilon}^\varepsilon x^a\eta(x)^2dx}\geq \frac{c}{\varepsilon^2}\min_{\eta\in C_0^\infty((-1,1))}\frac{\int_{-1}^1x^a\eta^\prime(x)^2dx}
{\int_{-1}^1x^a\eta(x)^2dx}\geq \frac{c}{\varepsilon^2}.\]
In particular, if $\varepsilon$ is small enough, $d>s$. Note that $\phi:=r^d\psi(\theta)$ is a positive $L_a$-harmonic function in the cone $\{|\theta|<2\varepsilon\}$. Moreover, since $\psi$ is even in $\theta$ (by the uniqueness of the first eigenfunction), $\phi$ is even in $y$.

For $\varepsilon$ sufficiently small, we have got a positive $L_a$-harmonic function $\phi$ in the cone $\{|y|\leq 2\varepsilon x\}$, satisfying $\phi\geq |z|^{2s}$ in $\mathcal{C}$. Apparently, $\partial_y^a|z|^s=\partial_y^a\phi=0$ on $\{y=0\}$. Then we can apply the maximum principle to
\[\frac{u-v-c(\varepsilon)|z|^s}{\phi},\]
to deduce that it is nonnegative in $\mathcal{C}$.
\end{proof}

\begin{prop}[{\bf Decay estimate}]\label{decay estimate}
For all $x>0$, $v(x,0)\leq C(1+x)^{-3s}$. For all $x<0$, $u(x)\leq C(1+|x|)^{-3s}$.
\end{prop}
\begin{proof}
For any $\lambda>0$ large, let
\[u^\lambda(x,y):=\lambda^{-s}u(\lambda x,\lambda y), \ \ \
v^\lambda(x,y):=\lambda^{-s}v(\lambda x,\lambda y).\]

By the previous lemma and Proposition \ref{upper bound},
\[u\geq c, \ \ \ v\leq C\ \ \ \mbox{in}\ B_{1/2}^+(1,0).\]

The equation for $v^\lambda$ is
\begin{equation*}
\left\{
\begin{aligned}
 &L_a v^\lambda=0, \ \mbox{in}\ B_{1/2}^+(1,0), \\
 &\partial^a_yv^\lambda=\lambda^{4s}\left(u^\lambda\right)^2v^\lambda\geq c\lambda^{4s}v^\lambda, \ \mbox{on}\ \partial B_{1/2}^+(1,0).
                          \end{aligned} \right.
\end{equation*}

By Lemma \ref{lem decay},
\[v^\lambda(1,0)\leq C\lambda^{-4s}.\]
This then gives the estimate for $v(\lambda,0)$.
\end{proof}

Before proving a similar decay estimate for $\frac{\partial
u}{\partial x}$ and $\frac{\partial v}{\partial x}$, we first give
an upper bound for the gradient of $u$ and $v$.

\begin{prop}\label{gradient upper bound}
There exists a constant $C$ such that,
 \[\Big|\frac{\partial u}{\partial x}(x,y)\Big|+\Big|\frac{\partial v}{\partial x}(x,y)\Big|\leq C(1+|x|+|y|)^{s-1}.\]
 \[\Big|y^a\frac{\partial u}{\partial y}(x,y)\Big|+\Big|y^a\frac{\partial v}{\partial x}(x,y)\Big|\leq C(1+|x|+|y|)^{-s}.\]
\end{prop}
\begin{proof}
For all $\lambda$ large, consider $(u^\lambda, v^\lambda)$ introduced in the proof of the previous
proposition. It satisfies
\begin{equation}\label{equation rescaled}
\left\{
\begin{aligned}
 &L_a u^\lambda=L_a v^\lambda=0, \ \mbox{in}\ B_1^+, \\
 &\partial_y^a u^\lambda=\lambda^{4s}u^\lambda\left(v^\lambda\right)^2,\ \
  \partial_y^a v^\lambda=\lambda^{4s}v^\lambda\left(u^\lambda\right)^2, \ \ \mbox{on}\ \partial^0 B_1^+.
                          \end{aligned} \right.
\end{equation}

By Proposition \ref{upper bound}, $u^\lambda$ and $v^\lambda$ are uniformly bounded in $\overline{B_1^+}$.
Then by the gradient estimate Theorem \ref{Yau's gradient estimate},
\[\sup_{\{y\geq |x|/2\}\cap(B_1^+\setminus B_{1/2}^+)}|\nabla u^\lambda|+|\nabla v^\lambda|\leq C.\]
Rescaling back this gives the claimed estimates in the part $\{y\geq |x|/2\}$. (Note that here $y^a$ is comparable to $(|x|+y)^a$.)

Next we consider the part $D:=\{0\leq y< x\}\cap(B_1^+\setminus B_{1/2}^+)$. Here by differentiating \eqref{equation rescaled} we obtain
\begin{equation*}
\left\{
\begin{aligned}
 &L_a \frac{\partial u^\lambda}{\partial x}=L_a \frac{\partial v^\lambda}{\partial x}=0, \ \mbox{in}\ D, \\
 &\partial_y^a \frac{\partial u^\lambda}{\partial x}=\lambda^{4s}\left(v^\lambda\right)^2\frac{\partial u^\lambda}{\partial x}
 +2\lambda^{4s}u^\lambda v^\lambda\frac{\partial v^\lambda}{\partial x},\ \ \mbox{on}\ \partial^0D,\\
 &\partial_y^a \frac{\partial v^\lambda}{\partial x}=\lambda^{4s}\left(u^\lambda\right)^2\frac{\partial v^\lambda}{\partial x}
 +2\lambda^{4s}u^\lambda v^\lambda\frac{\partial u^\lambda}{\partial x},\ \ \mbox{on}\ \partial^0D.
                          \end{aligned} \right.
\end{equation*}
By Corollary \ref{bound on energy},
\[\int_Dy^a\left(\Big|\frac{\partial u^\lambda}{\partial x}\Big|^2+\Big|\frac{\partial v^\lambda}{\partial x}\Big|^2\right)\leq C,\]
for a constant $C$ independent of $\lambda$.

By Proposition \ref{decay estimate}, $v^\lambda\leq C\lambda^{-4s}$ on $\partial^0D$. Thus the coefficient $2\lambda^{4s}u^\lambda v^\lambda$ is uniformly bounded on $\partial^0D$. Although $\lambda^{4s}\left(u^\lambda\right)^2$ is not uniformly bounded, it has a favorable sign. Then standard Moser iteration (see for example \cite[Theorem 1.2]{TX}) gives
\[\sup_{\{0\leq y< x/2\}\cap(B_{3/4}^+\setminus B_{2/3}^+)}\Big|\frac{\partial u^\lambda}{\partial x}\Big|+\Big|\frac{\partial v^\lambda}{\partial x}\Big|\leq C,\]
for a constant $C$ independent of $\lambda$.

Finally, similar to the proof of Lemma \ref{lem limit equation 1}, we have
\begin{equation*}
\left\{
\begin{aligned}
 &L_{-a} \left(y^a\frac{\partial u^\lambda}{\partial y}\right)=L_{-a} \left(y^a\frac{\partial v^\lambda}{\partial y}\right)=0, \ \mbox{in}\ D, \\
 &y^a\frac{\partial u^\lambda}{\partial y}=\lambda^{4s}u^\lambda\left(v^\lambda\right)^2\in(0, C),\ \ \mbox{on}\ \partial^0D,\\
 &y^a\frac{\partial v^\lambda}{\partial y}=\lambda^{4s}v^\lambda\left(u^\lambda\right)^2\in(0, C),\ \ \mbox{on}\ \partial^0D.
                          \end{aligned} \right.
\end{equation*}
Moreover, by Corollary \ref{bound on energy},
\[\int_Dy^{-a}\left(\Big|y^a\frac{\partial u^\lambda}{\partial y}\Big|^2+\Big|y^a\frac{\partial v^\lambda}{\partial y}\Big|^2\right)\leq C,\]
for a constant $C$ independent of $\lambda$.

Then by applying the Moser iteration to $(y^a\frac{\partial u^\lambda}{\partial y}-C)_+$ and $(y^a\frac{\partial u^\lambda}{\partial y}-C)_-$,
we see
\[\sup_{\{0\leq y< x/2\}\cap(B_{3/4}^+\setminus B_{2/3}^+)}\Big|y^a\frac{\partial u^\lambda}{\partial y}\Big|+\Big|y^a\frac{\partial v^\lambda}{\partial y}\Big|\leq C,\]
for a constant $C$ independent of $\lambda$.
\end{proof}
Written in polar coordinates, this reads as
\begin{coro}\label{coro gradient upper bound}
There exists a constant $C$ such that,
 \[\Big|\frac{\partial u}{\partial r}(r,\theta)\Big|+\Big|\frac{\partial v}{\partial r}(r,\theta)\Big|\leq C(1+r)^{s-1}.\]
 \[\Big|\frac{\left(\sin\theta\right)^a}{r}\frac{\partial u}{\partial \theta}(r,\theta)\Big|+\Big|\frac{\left(\sin\theta\right)^a}{r}\frac{\partial u}{\partial \theta}(r,\theta)\Big|\leq C(1+r)^{s-1}.\]
\end{coro}
\begin{proof}
We have
\[\frac{\partial u}{\partial r}=\cos\theta\frac{\partial u}{\partial x}+\left(\sin\theta\right)^{1-a}\left(\sin\theta\right)^{a}\frac{\partial u}{\partial y},\]
\[\frac{\left(\sin\theta\right)^a}{r}\frac{\partial u}{\partial \theta}(r,\theta)=-\left(\sin\theta\right)^{1+a}\frac{\partial u}{\partial x}+\cos\theta\left(\sin\theta\right)^{a}\frac{\partial u}{\partial y}.\]
Since $1+a>0$ and $1-a>0$, $\left(\sin\theta\right)^{1-a}$ and $\left(\sin\theta\right)^{1+a}$ are bounded. Then this corollary follows from the previous proposition.
\end{proof}

Now we can give a further decay estimate for $\frac{\partial v}{\partial x}$ and $\frac{\partial u}{\partial x}$.
\begin{prop}\label{gradient decay estimate}
For all $x>0$, $\big|\frac{\partial v}{\partial x}(x,0)\big|\leq C(1+x)^{-3s-1}$. For all $x<0$, $\big|\frac{\partial u}{\partial x}(x,0)\big|\leq C(1+|x|)^{-3s-1}$.
\end{prop}
\begin{proof}
We use notations introduced in the proof of Proposition \ref{decay estimate}.

By differentiating the equation for $v^\lambda$, we obtain
\begin{equation*}
\left\{
\begin{aligned}
 &L_a \left(\frac{\partial v^\lambda}{\partial x}\right)_+\geq0, \ \mbox{in}\ B_{1/2}^+(1,0), \\
 &\partial_y^a\left(\frac{\partial v^\lambda}{\partial x}\right)_+\geq\lambda^{4s}\left(u^\lambda \right)^2
 \left(\frac{\partial v^\lambda}{\partial x}\right)_+-2\lambda^{4s}u^\lambda v^\lambda \frac{\partial u^\lambda}{\partial x}
 \left(\frac{\partial v^\lambda}{\partial x}\right)_+, \ \mbox{on}\ \partial B_{1/2}^+(1,0).
                          \end{aligned} \right.
\end{equation*}

By Proposition \ref{decay estimate}, $v^\lambda\leq C\lambda^{-4s}$ on $\partial^0B_{1/2}^+(1,0)$. By the previous proposition $|\frac{\partial u^\lambda}{\partial x}|\leq C$ in $\overline{B_{1/2}^+(1,0)}$. Lemma \ref{sharp lower bound 3} also implies that $u^\lambda\geq c$ in $\overline{B_{1/2(1,0)}^+}$.
 Hence  on $\partial^0 B_{1/2}^+(1,0)$,
\[\partial_y^a\left(\frac{\partial v^\lambda}{\partial x}\right)_+
\geq \left(c\lambda^{4s}-C\right)\left(\frac{\partial
v^\lambda}{\partial x}\right)_+.\]  Applying Lemma \ref{lem decay},
we get
\[\frac{\partial v^\lambda}{\partial x}(1,0)\leq C\lambda^{-4s}.\]
The same estimate holds for the negative part. This then implies the
bound for $\big|\frac{\partial v}{\partial x}(\lambda,0)\big|$.
\end{proof}

\section{Refined asymptotics at infinity}
\setcounter{equation}{0}

In this section we prove a refined asymptotic expansion of the solution $(u, v)$. See Proposition \ref{convergence rate improved} below. Here we need $s>\frac{1}{4}$.  The refined asymptotic is needed for the method of moving planes in the next section.

\subsection{}

Let
\[x=e^t\cos\theta,\ \ \ y=e^t\sin\theta, \ \ \ t\in\R, \ \theta\in[0,\pi]\]
and
\[\bar{u}(t,\theta)=e^{-st}u(e^t\cos\theta, e^t\sin\theta),\ \ \ \bar{v}(t,\theta)=e^{-st}v(e^t\cos\theta, e^t\sin\theta).\]
The equation \eqref{equation} can be transformed to the one for $(\bar{u},\bar{v})$,
\begin{equation}\label{Emden-Fowler}
\left\{
\begin{aligned}
 &\bar{u}_{tt}+\bar{u}_t+s(1-s)\bar{u}+L^a_\theta\bar{u}=0, \ \mbox{in}\ (-\infty,+\infty)\times(0,\pi), \\
 &\bar{v}_{tt}+\bar{v}_t+s(1-s)\bar{v}+L^a_\theta\bar{v}=0, \ \mbox{in}\ (-\infty,+\infty)\times(0,\pi), \\
 &\lim_{\theta\to 0\ or \ \pi}\partial_\theta^a\bar{u}=\pm e^{(4-a)st}\bar{u}\bar{v}^2, \ \mbox{on}\ (-\infty,+\infty)\times\{0,\pi\},\\
  &\lim_{\theta\to 0\ or \ \pi}\partial_\theta^a\bar{v}=\pm e^{(4-a)st}\bar{v}\bar{u}^2, \ \mbox{on}\ (-\infty,+\infty)\times\{0,\pi\},
                          \end{aligned} \right.
\end{equation}
where we take the positive sign $+$ at $\{0\}$ and the negative one $-$ at $\{\pi\}$.

By Proposition \ref{upper bound},
\begin{equation}\label{uniform upper bound}
0\leq\bar{u},\ \bar{v}\leq C,\  \ \mbox{in}\ [1,+\infty)\times[0,\pi].
\end{equation}
By Proposition \ref{decay estimate},
\begin{equation}\label{boundary condition 1}
\left\{
\begin{aligned}
 &\bar{u}\leq Ce^{-4st}, \ \mbox{on}\ [1,+\infty)\times\{\pi\},\\
  &\bar{v}\leq Ce^{-4st}, \ \mbox{on}\ [1,+\infty)\times\{0\}.
                          \end{aligned} \right.
\end{equation}
Combining Proposition \ref{upper bound} and Proposition \ref{decay estimate}, we also have
\begin{equation}\label{boundary condition 2}
\left\{
\begin{aligned}
 &0\leq\partial_\theta^a\bar{u}\leq C e^{-(4+a)st}, \ \mbox{on}\ [1,+\infty)\times\{0\}\\
  &0\geq\partial_\theta^a\bar{v}\geq-C e^{-(4+a)st}, \ \mbox{on}\ [1,+\infty)\times\{\pi\}.
                          \end{aligned} \right.
\end{equation}

In Remark \ref{rmk 1}, we have shown that
\begin{lem}
As $t\to+\infty$, $\bar{u}(t,\theta)\to (\cos\frac{\theta}{2})^{2s}$ and $\bar{v}(t,\theta)\to (\sin\frac{\theta}{2})^{2s}$ uniformly in $[0,\pi]$.
\end{lem}

The next task is to get an exact convergence rate.
\begin{prop}\label{convergence rate}
There exists a constant $C$ so that
\[|u(r,\theta)-r^s(\cos\frac{\theta}{2})^{2s}|+|v(r,\theta)-r^s(\sin\frac{\theta}{2})^{2s}|\leq C(1+r)^{s-\min\{1,4s\}}.\]
\end{prop}
In the following we denote
\[\sigma:=\min\{1,4s\}.\]
\begin{proof}
Let $\phi(t,\theta):=\bar{u}(t,\theta)-(\cos\frac{\theta}{2})^{2s}$. There exists a constant $M$ such that
\begin{equation}\label{Emden-Fowler 2}
\left\{
\begin{aligned}
 &\phi_{tt}+\phi_t+s(1-s)\phi+L_\theta^a\phi=0, \ \mbox{in}\ [1,+\infty)\times(0,\pi), \\
 &0\leq\phi(t,\pi)\leq Me^{-4st},\\
  &0\leq\partial_\theta^a\phi(t,0)\leq Me^{-(4+a)st}.
                          \end{aligned} \right.
\end{equation}
Moreover, $|\phi|\leq M$ in $[1,+\infty)\times[0,\pi]$ and it converges to $0$ uniformly as $t\to+\infty$.

Let $\psi(t,\theta):=\left(\phi(t,\theta)-Me^{-4st}\right)_+$. It satisfies
\begin{equation}\label{Emden-Fowler 3}
\left\{
\begin{aligned}
 &\psi_{tt}+\psi_t+s(1-s)\psi+L_\theta^a\psi\geq-s(1-s)Me^{-4st}, \ \mbox{in}\ [1,+\infty)\times(0,\pi), \\
 &\psi(t,\pi)=0,\\
  &\partial_\theta^a\psi(t,0)\geq 0.
                          \end{aligned} \right.
\end{equation}
We still have $0\leq\psi\leq M$ in $[1,+\infty)\times[0,\pi]$, and it converges to $0$ uniformly as $t\to+\infty$.

Define
$$f(t):=\int_0^\pi\psi(t,\theta)\left(\cos\frac{\theta}{2}\right)^{2s}\left(\sin\theta\right)^ad\theta\geq 0.$$
Multiplying \eqref{Emden-Fowler 3} by $\left(\cos\frac{\theta}{2}\right)^{2s}$ and integrating on $(0,\pi)$ with respect to the measure $\left(\sin\theta\right)^ad\theta$, we obtain
\begin{equation}\label{7.1}
f^{\prime\prime}(t)+f^\prime(t)\geq f^{\prime\prime}(t)+f^\prime(t)-\partial^a_\theta\psi(t,0)\geq-Ce^{-4st}.
\end{equation}

This implies that
\[\left(e^tf^\prime(t)+\frac{C}{1-4s}e^{(1-4s)t}\right)^\prime\geq0.\]
Consequently,
\[e^tf^\prime(t)+\frac{C}{1-4s}e^{(1-4s)t}\geq ef^\prime(1)+\frac{C}{1-4s}e^{1-4s}\geq -C, \ \ \forall \ t\geq 1.\]
Hence we have
\[f^\prime(t)\geq -Ce^{-t}-\frac{C}{1-4s}e^{-4st}\ \ \ \mbox{on}\ [1,+\infty).\]

Integrating from $t$ to $+\infty$, we obtain
\[f(t)\leq\int_t^{+\infty}\left(f^{\prime}(s)\right)_-ds\leq Ce^{-\sigma t}, \ \ \forall \ t\in[1,+\infty).\]
A similar estimate holds for $[\bar{u}(t,\theta)-(\cos\frac{\theta}{2})^{2s}+Me^{-4st}]_-$.

Now we have got, for all $t\geq1$,
\[\int_0^\pi\Big|\bar{u}(t,\theta)-\left(\cos\frac{\theta}{2}\right)^{2s}\Big|
\left(\cos\frac{\theta}{2}\right)^{2s}\left(\sin\theta\right)^ad\theta\leq Ce^{-\sigma t}.\]
Then by standard estimates we get, for any $h>0$,
\[\sup_{\theta\in(0,\pi-h)}\big|\bar{u}(t,\theta)-\left(\cos\frac{\theta}{2}\right)^{2s}\big|\leq \frac{C}{h}e^{-\sigma t}.\]

Next we extend this bound to $(\pi-h,\pi)$. Let
$$\varphi:=\left(\phi-\max\{\frac{C}{h}e^{-\sigma t}, Me^{-4st}\}\right)_+,\ \ \ \mbox{on} \ [1,+\infty)\times[\pi-h,\pi].$$
It satisfies
\begin{equation}\label{Emden-Fowler 4}
\left\{
\begin{aligned}
 &\varphi_{tt}+\varphi_t+s(1-s)\varphi+L_\theta^a\varphi\geq-Ce^{-\sigma t}, \ \mbox{in}\ (1,+\infty)\times(\pi-h,\pi), \\
 &\varphi(t,\pi)=\varphi(t,\pi-h)=0.
                          \end{aligned} \right.
\end{equation}
Let
\[g(t):=\int_{\pi-h}^\pi\varphi(t,\theta)^2\left(\sin\theta\right)^ad\theta.\]
We claim that the following Poincare inequality holds.

{\bf Claim.} There exists a constant $c$ so that \[\frac{\int_{\pi-h}^\pi\left(\frac{\partial\varphi}{\partial\theta}(t,\theta)\right)^2\left(\sin\theta\right)^ad\theta}
{\int_{\pi-h}^\pi\varphi(t,\theta)^2\left(\sin\theta\right)^ad\theta}
\geq\frac{c}{h^2}.\]
This is because the left hand side can be bounded from below by
\[c\min_{\eta\in C_0^\infty((0,h))}\frac{\int_0^hx^a\eta^\prime(x)^2dx}
{\int_0^hx^a\eta(x)^2dx}\geq \frac{c}{h^2}\min_{\eta\in C_0^\infty((0,1))}\frac{\int_0^1x^a\eta^\prime(x)^2dx}
{\int_0^1x^a\eta(x)^2dx}.\]

Multiplying \eqref{Emden-Fowler 4} by $\varphi\left(\sin\theta\right)^a$ and integrating on $(\pi-h,\pi)$ leads to
\[g^{\prime\prime}(t)+g^\prime(t)-\left[\frac{c}{h^2}-s(1-s)\right]g(t)\geq -Ce^{-\sigma t}g(t)^{\frac{1}{2}}\geq
-Ce^{-2\sigma t}-s(1-s)g(t).\]
Thus
\[g^{\prime\prime}(t)+g^\prime(t)-\left[\frac{c}{h^2}-2s(1-s)\right]g(t)\geq
-Ce^{-2\sigma t}.\]
Now we fix an $h$ small so that
\[\frac{c}{h^2}-2s(1-s)>4\sigma^2 .\]
Because $g(t)\to0$ as $t\to+\infty$, by the comparison principle,
\[g(t)\leq C\left(e^{-2\sigma t}+e^{-\frac{1+\sqrt{1+4\left(\frac{c}{h^2}-2s(1-s)\right)}}{2}t}\right)\leq Ce^{-2\sigma t}.\]
Then standard elliptic estimates imply that
\begin{eqnarray*}
\sup_{[\pi-h,\pi]}\Big|\bar{u}(t,\theta)-\left(\cos\frac{\theta}{2}\right)^{2s}\Big|&\leq& \max\{\frac{C}{h}e^{-\sigma t}, Me^{-4st}\}+Ce^{-\sigma t}\\
&\leq& Ce^{-\sigma t}.
\end{eqnarray*}

Coming back to $u$ this gives the claimed estimate.
\end{proof}

\subsection{}
Now consider $\bar{u}_t$. By differentiation in $t$, $\bar{u}_t$ still satisfies the equation in \eqref{Emden-Fowler}.
Moreover, we have the following boundary conditions. At $\theta=\pi$, by Proposition \ref{decay estimate} and Proposition \ref{gradient decay estimate},
\[
\bar{u}_t(t,\pi)=-se^{-st}u(e^t,\pi)+e^{(1-s)t}u_r(e^t,\pi)=O(e^{-4st}).\]

Note that $\bar{u}_t(t,0)$ is bounded in $t\in[1,+\infty)$, which can be deduced from Proposition \ref{decay estimate} and Proposition \ref{gradient upper bound}. At $\theta=0$, because $|\bar{v}(t,0)|+|\bar{v}_t(t,0)|\leq Ce^{-4st}$,
\begin{eqnarray*}
\partial_\theta^a\bar{u}_t&=&\left(4-a\right)se^{(4-a)st}\bar{u}\bar{v}^2+e^{(4-a)st}\bar{v}^2\bar{u}_t
+2e^{(4-a)st}\bar{u}\bar{v}\bar{v}_t\\
&=&O(e^{-(4+a)st}).
\end{eqnarray*}

Then as in the proof of the previous section, we have
\begin{lem}\label{lem bound on ut}
As $t\to+\infty$,
\[\sup_{\theta\in[0,\pi]}\big|\bar{u}_t(t,\theta)\big|+\big|\bar{v}_t(t,\theta)\big|\leq Ce^{-\min\{1,4s,(4+a)s\}t}.\]
\end{lem}

\subsection{}
Now we assume $s>1/4$. This implies that $(4+a)s>1$, $4s>1$ and hence $\sigma=1$.
Here we improve Proposition \ref{convergence rate} to
\begin{prop}\label{convergence rate improved}
There exist two constants $a$ and $b$ so that we have the expansion
\[u(r,\theta)=r^s(\cos\frac{\theta}{2})^{2s}+ar^{s-1}(\cos\frac{\theta}{2})^{2s}+o(r^{s-1}),\]
\[v(r,\theta)=r^s(\sin\frac{\theta}{2})^{2s}+br^{s-1}(\sin\frac{\theta}{2})^{2s}+o(r^{s-1}).\]
\end{prop}
\begin{proof}
Let
\[\tilde{u}(t,\theta):=e^{t}\left[\bar{u}(t,\theta)-\left(\cos\frac{\theta}{2}\right)^{2s}\right],\]
and $\tilde{v}$ be defined similarly.

By Proposition \ref{convergence rate}, $\tilde{u}$ is bounded on $[1,+\infty)\times[0,\pi]$. Moreover,
\[\tilde{u}_t(t,\theta)=e^{t}\left[\bar{u}(t,\theta)-\left(\cos\frac{\theta}{2}\right)^{2s}\right]+e^t\bar{u}_t(t,\theta),\]
is also uniformly bounded, thanks to the estimate in Lemma \ref{lem bound on ut}.

$\tilde{u}$ satisfies
\begin{equation}\label{Emden-Fowler 5}
\left\{
\begin{aligned}
 &\tilde{u}_{tt}-\tilde{u}_t+s\left(1-s\right)\tilde{u}+L_\theta^a
 \tilde{u}=0, \ \mbox{in}\ [1,+\infty)\times(0,\pi), \\
  &0\leq\partial_\theta^a\tilde{u}(t,0)\leq Me^{-[(4+a)s-1]t},\\
  &0\leq\tilde{u}(t,\pi)\leq Me^{-(4s-1)t}.
                          \end{aligned} \right.
\end{equation}

The exact boundary conditions are as follows:
\begin{equation}\label{boundary condition 3}
\left\{
\begin{aligned}
  &\partial_\theta^a\tilde{u}(t,0)=e^{[(4-a)s-1]t}\tilde{v}^2\left(1+e^{-t}\tilde{u}\right),\\
  &\partial^a_\theta\tilde{u}(t,\pi)=-2^ase^{t}-e^{(4-a)st}\tilde{u}\left(1+e^{- t}\tilde{v}\right)^2.
                          \end{aligned} \right.
\end{equation}
Note that although some terms (or coefficients before $\tilde{u}$) in these boundary conditions are not bounded when $t\to+\infty$, they all have a favorable sign. This allows us to apply the main result in \cite{TVZ 2} to deduce that $\tilde{u}$ and $\tilde{v}$ are bounded in $C^\beta([1,+\infty)\times[0,\pi])$ for some $\beta>0$.

Thus for any $t_i\to +\infty$, we can assume that $\tilde{u}(t_i+t,\theta)$ converges to a limit function $\tilde{u}^\infty$ in $C_{loc}(\R\times[0,\pi])$. Then $\tilde{u}^\infty$ satisfies
\begin{equation}\label{9.3}
\left\{
\begin{aligned}
 &\tilde{u}^\infty_{tt}+\tilde{u}^\infty_t+s(1-s)\tilde{u}^\infty+L_\theta^a\tilde{u}^\infty=0, \ \mbox{in}\ \R\times(0,\pi), \\
  &\big|\tilde{u}^\infty(t,\theta)\big|\leq C,\\
  &\tilde{u}^\infty(t,\pi)=0, \\
  &\partial^a_\theta\tilde{u}^\infty(t,0)=0.
                          \end{aligned} \right.
\end{equation}

Consider the eigenvalue problem
\begin{equation*}
\left\{
\begin{aligned}
 &-L_\theta^a\psi_j=\lambda_j\psi_j, \ \mbox{in}\ (0,\pi), \\
   &\psi_j(\pi)=0, \\
  &\partial^a_\theta\psi_j(0)=0.
  \end{aligned} \right.
\end{equation*}
This problem has a sequence of eigenvalues $\lambda_1<\lambda_2\leq \cdots\leq \lambda_k\to+\infty$, and the corresponding eigenfunctions are denoted by $\psi_j$, which is normalized in $L^2((0,\pi),(\sin\theta)^ad\theta)$. Here the first eigenvalue $\lambda_1=s(1-s)$ is simple and $\psi_1(\theta)=(\cos\frac{\theta}{2})^{2s}$ (modulo a constant) is positive in $(0,\pi)$.

Consider the decomposition
\[\tilde{u}^\infty(t,\theta)=\sum_{j=1}^\infty c_j(t)\psi_j(\theta).\]
Then $c_j(t)$ satisfies
\[c_j^{\prime\prime}+c_j^\prime+\left[s\left(1-s\right)-\lambda_j\right]c_j=0.\]
Note that $|c_j(t)|\leq C$ for all $t$. Combined with the above equation, we see $c_j\equiv 0$ for all $j\geq 2$, and $c_1(t)\equiv \tilde{a}$ for some constant $\tilde{a}$.

Now we show that this constant $\tilde{a}$ does not depend on the sequence $t_i\to+\infty$.
In the following we denote
\[\delta:=\min\{(4+a)s-1,4s-1\}.\]
Let
\[f(t):=\int_0^\pi \tilde{u}(t,\theta)\left(\cos\frac{\theta}{2}\right)^{2s}\left(\sin\theta\right)^ad\theta.\]
By the bound on $\tilde{u}$ and $\tilde{u}_t$, $f(t)$ and
\[f^\prime(t)=\int_0^\pi \tilde{u}_t(t,\theta)\left(\cos\frac{\theta}{2}\right)^{2s}\left(\sin\theta\right)^ad\theta\]
are bounded on $[1,+\infty)$. Multiplying the equation in \eqref{Emden-Fowler 5} by $\left(\cos\frac{\theta}{2}\right)^{2s}\left(\sin\theta\right)^a$ and integrating by parts leads to
\[f^{\prime\prime}(t)-f^\prime(t)=-\partial_\theta^a\tilde{u}(t,0)-2^as\tilde{u}(t,\pi)=O(e^{-\delta t}).\]
In particular, $f^{\prime\prime}(t)$ is also bounded on $[1,+\infty)$

For any $t_i\to+\infty$, we can assume that $f(t_i+t)$ converges to a limit $f_\infty(t)$ in $C^1_{loc}(\R)$, which satisfies
\[f_\infty^{\prime\prime}(t)-f_\infty^\prime(t)=0.\]
Because $f_\infty$ is bounded on $\R$, it must be a constant. Thus $f_\infty^\prime\equiv 0$.
This implies that $f^\prime(t)\to 0$ as $t\to+\infty$.

Now we also have
\[\left(e^{-t}f^\prime(t)\right)^\prime=O(e^{-(1+\delta)t})\]
Integrating this on $[t,+\infty)$, we obtain
\[\big|f^\prime(t)\big|=O(e^{-\delta t}).\]
Hence there exists a constant $a$ such that
\[\big|f(t)-a\big|=O(e^{-\delta t}).\]

Together with the previous analysis, we see for any $t\to+\infty$,
\[\tilde{u}(t,\theta)\to a\left(\cos\frac{\theta}{2}\right)^{2s},\]
uniformly in $C([0,\pi])$. This gives the expansion of $u$.
\end{proof}
\begin{rmk}
We can also estimate the convergence rate of $\tilde{u}$, which is of order $O(e^{-\delta t})$. Hence in the expansion of $u$,
$o(r^{s-1})$ can be replaced by $O(r^{s-1-\delta})$.
\end{rmk}

\section{Moving plane argument}
\setcounter{equation}{0}

In this section, we prove
\begin{thm}\label{lem uniqueness}
Let $(u_i,v_i)$, $i=1,2$ be two solutions of \eqref{equation extension}. Suppose that they satisfy
\begin{equation}\label{fine asymptotics 1}
u_i(r,\theta)=r^s\left(\cos\frac{\theta}{2}\right)^{2s}+a_ir^{s-1}\left(\cos\frac{\theta}{2}\right)^{2s}+o(r^{s-1}),\ \ \
i=1,2,
\end{equation}
\begin{equation}\label{fine asymptotics 2}
v_i(r,\theta)=r^s\left(\sin\frac{\theta}{2}\right)^{2s}+b_ir^{s-1}\left(\sin\frac{\theta}{2}\right)^{2s}+o(r^{s-1}),\ \ \
i=1,2,
\end{equation}
for four constants $a_i,b_i, i=1,2$. If $a_1+b_1=a_2+b_2$, then
\[u_1(x+t_0,y)\equiv u_2(x,y),\ \ \ \ v_1(x+t_0,y)\equiv v_2(x,y),\]
where $t_0=\frac{1}{s}(a_2-a_1)=\frac{1}{s}(b_1-b_2)$.
\end{thm}

%After a translation of $(u_1,v_1)$ in the $x$-direction, we may assume $a_2=a_1=a$ and
%consequently $b_2=b_1=b$ and $t_0=0$.

Note that \eqref{fine asymptotics 1} and \eqref{fine asymptotics 2} imply that
\begin{equation}\label{fine asymptotics 3}
\big|u_i(r,\theta)-r^s\left(\cos\frac{\theta}{2}\right)^{2s}\big|
+\big|v_i(r,\theta)-r^s\left(\sin\frac{\theta}{2}\right)^{2s}\big|\leq Mr^{s-1},\ \ \
i=1,2,
\end{equation}
for some constant $M>0$.

For any $t\in\R$, let
\[u^t(x,y):=u_1(x+t,y),\ \ \ \ v^t(x,y):=v_1(x+t,y),\]
which is still a solution of \eqref{equation extension}.

In the following, it will be helpful to keep the following fact in mind. Because
\begin{eqnarray*}
\big|u^t-u_2\big|&\leq& M\left(1+|z|\right)^{s-1}+M\left(1+|z+t|\right)^{s-1}\\
&&  +\Big|\left(\frac{\sqrt{x^2+y^2}+x}{2}\right)^{2s}-\left(\frac{\sqrt{(x+t)^2+y^2}+x+t}{2}\right)^{2s}\Big|,
\end{eqnarray*}
$\big|u^t-u_2\big|\to0$ as $z\to\infty$. Thus any positive maximum (or negative minima) of $u^t-u_2$ is attained at some point.

The first step is to show that we can start the moving plane from the infinity.
\begin{lem}\label{lem 8.1}
 If $t$ is large enough,
\begin{equation}\label{8.1}
u^t(x,y)\geq u_2(x,y),\ \ \ \ v^t(x,y)\geq v_2(x,y),\ \ \ \mbox{on}\ \overline{\R_+^2}.
\end{equation}
\end{lem}
\begin{proof}
If $t$ is sufficiently large, for $x\geq 0$,
\begin{eqnarray*}
u^t(x,0)&\geq& \left(x+t\right)^s-M\left(x+t\right)^{s-1}\\
&\geq& x^s+Mx^{s-1}\\
&\geq& u_2(x,0).
\end{eqnarray*}
Similarly,
\[v^t(x,0)\leq v_2(x,0), \ \ \ \mbox{on}\ (-\infty,-C(M)].\]

It can be checked directly that for $t$ large,
\[u^t(x,0)\geq u_2(x,0),\ \ \ \ \mbox{on}\ [-C(M),0].\]
In fact, for $x\in [-C(M),0]$, $\lim_{t\to+\infty}u^t(x,0)=+\infty$
uniformly (see Lemma \ref{lem 7.2}).

Then by noting that
\begin{equation*}
\left\{
\begin{aligned}
 &L_a\left(u^t-u_2\right)=0, \ \mbox{in}\ \R^2_+, \\
  &\big|u^t(z)-u_2(z)\big|\to 0, \ \mbox{as}\ z\in\R^2_+,\ z\to\infty, \\
   &u^t(x,0)-u_2(x,0)\geq0, \ \mbox{on}\ [-C(M),+\infty) \\
 &\partial^a_y\left(u^t(x,0)-u_2(x,0)\right)\leq
 v_2(x,0)^2\left(u^t(x,0)-u_2(x,0)\right) \ \mbox{on}\
 (-\infty,-C(M)),
                          \end{aligned} \right.
\end{equation*}
we can apply the maximum principle to deduce that
\[u^t\geq u, \ \mbox{in}\ \R^2_+.\]

The same reasoning using
\begin{equation*}
\left\{
\begin{aligned}
 &L_a\left(v^t-v_2\right)=0, \ \mbox{in}\ \R^2_+, \\
  &\big|v^t(z)-v_2(z)\big|\to 0, \ \mbox{as}\ z\in\R^2_+,\ z\to\infty, \\
   &\partial^a_y\left(v^t(x,0)-v_2(x,0)\right)\geq
 u^t(x,0)^2\left(v^t(x,0)-v_2(x,0)\right) \ \mbox{on}\
 \partial\R^2_+,
                          \end{aligned} \right.
\end{equation*}
gives
\[v^t\leq v, \ \mbox{in}\ \R^2_+.\qedhere\]
\end{proof}

Now we can define $t_0$ to be
\begin{equation}\label{8.2}
\min\{t: \ \forall s>t, u^s(x,y)\geq u_2(x,y),\ \ \ \ v^s(x,y)\geq v_2(x,y),\ \ \ \mbox{on}\ \overline{\R_+^2}\}.
\end{equation}
By continuity, $u^{t_0}\geq u_2$, $v^{t_0}\leq v_2$.

We want to prove that $t_0=\frac{1}{s}(a_2-a_1)$. Indeed, if this is true, we have
$u^{t_0}\geq u_2$ and $v^{t_0}\leq v_2$. Then we can slide from the
left, by the same reasoning this procedure must stop at $t_0$. Thus we also have $u^{t_0}\leq u_2$ and $v^{t_0}\geq v_2$. Consequently $u^{t_0}\equiv u_2$
and $v^{t_0}\equiv v_2$.

Now assume $t_0>\frac{1}{s}(a_2-a_1)$. We will get a contradiction from this.
Let $\delta_0=st_0-(a_2-a_1)>0$. By \eqref{fine asymptotics 1},
\[u^{t_0}(x,0)=x^s+\left(a_1+st_0\right)x^{s-1}+o(x^{s-1}),
\ \ \ \mbox{as}\ x\to+\infty.\]
\[v^{t_0}(x,0)=|x|^s+\left(b_1-st_0\right)|x|^{s-1}+o(|x|^{s-1}),
\ \ \ \mbox{as}\ x\to-\infty.\]
Comparing with $u_2$ and $v_2$ respectively, we get a constant $T_0$ such that
\begin{equation}\label{8.4}
u^{t_0}(x,0)\geq u_2(x,0)+\frac{\delta_0}{2}x^{s-1},\ \ \ \mbox{if}\ x\geq T_0,
\end{equation}
and
\begin{equation}\label{8.5}
v^{t_0}(x,0)\leq v_2(x,0)-\frac{\delta_0}{2}|x|^{s-1},\ \ \ \mbox{if}\
x\leq -T_0.
\end{equation}
By \eqref{fine asymptotics 1}, perhaps after choosing a larger $T_0$, for all $t\geq 0$ we have
\begin{equation*}
\big|u^{t}(x,0)-x^s-\left(a_1+st\right)x^{s-1}\big|\leq
\frac{\delta_0}{8}x^{s-1},\ \ \ \mbox{if}\ x\geq T_0.
\end{equation*}
Thus there exists an $\varepsilon_1>0$ such that, for all
$t\in[t_0-\varepsilon_1,t_0]$,
\[
u^{t}(x,0)\geq u^{t_0}(x,0)-\frac{\delta_0}{4}x^{s-1},\ \ \ \mbox{if}\ x\geq T_0.
\]
Combining this with \eqref{8.4}, we see for these $t$,
\begin{equation}\label{11.1}
u^{t}(x,0)\geq u_2(x,0),\ \ \ \mbox{if}\ x\geq T_0,
\end{equation}
and similarly
\begin{equation}\label{11.2}
v^{t}(x,0)\leq v_2(x,0),\ \ \ \mbox{if}\ x\leq -T_0,
\end{equation}

 By the strong
maximum principle,  $u^{t_0}>u_2$ and $v^{t_0}<v_2$ strictly. In fact, if there exists a point $z_0\in\R^2_+$ such that $u^{t_0}(z_0)=u_2(z_0)$, then the strong maximum principle implies that $u^{t_0}\equiv u_2$, which contradicts \eqref{8.4}.

Next, by continuity we can find an $\varepsilon_2>0$ so that for all $t\in[t_0-\varepsilon_2,t_0]$,
\[u^{t}(x,0)\geq u_2(x,0),\ \ \ v^{t}(x,0)\leq v_2(x,0),\ \ \ \mbox{for}\ x\in[-T_0,T_0].\]

Combined with \eqref{11.1} and \eqref{11.2}, by choosing $\varepsilon:=\min\{\varepsilon_1,\varepsilon_2\}$, we see for all $t\in[t_0-\varepsilon,t_0]$,
\[u^{t}(x,0)-u_2(x,0)\geq0,\ \ \ \mbox{in}\ [-T_0,+\infty),\]
\[v^{t}(x,0)-v_2(x,0)\leq0,\ \ \ \mbox{in}\ (-\infty,T_0].\]
Then arguing as in the proof of Lemma \ref{lem 8.1}, we know
for all $t\in[t_0-\varepsilon,t_0]$,
\[u^{t}\geq u_2,\ \ v^{t}\leq v_2,\ \ \ \mbox{in}\ \overline{\R^2}.\]
However, this contradicts the definition of $t_0$. Thus the assumption $t_0>\frac{1}{s}(a_2-a_1)$ cannot be true.

\begin{proof}[Proof of Theorem \ref{main result}]
We first prove the symmetry between $u$ and $v$. Given a solution $(u,v)$ of \eqref{equation extension}, let
$(u_1(x,y),v_1(x,y))=(u(x,y),v(x,y))$ and $(u_2(x,y),v_2(x,y))=(v(-x,y),u(-x,y))$.

By Proposition \ref{convergence rate improved}, after a scaling, we have the expansion
\begin{equation}\label{12.1}
\left\{
\begin{aligned}
 &u_1(r,\theta)=u(r,\theta)=r^s(\cos\frac{\theta}{2})^{2s}+ar^{s-1}(\cos\frac{\theta}{2})^{2s}+o(r^{s-1}), \\
  &v_1(r,\theta)=v(r,\theta)=r^s(\sin\frac{\theta}{2})^{2s}+br^{s-1}(\sin\frac{\theta}{2})^{2s}+o(r^{s-1}).
                          \end{aligned} \right.
\end{equation}
Hence
\begin{equation*}
\left\{
\begin{aligned}
 &u_2(r,\theta)=r^s(\cos\frac{\theta}{2})^{2s}+br^{s-1}(\cos\frac{\theta}{2})^{2s}+o(r^{s-1}), \\
  &v_2(r,\theta)=r^s(\sin\frac{\theta}{2})^{2s}+ar^{s-1}(\sin\frac{\theta}{2})^{2s}+o(r^{s-1}).
                          \end{aligned} \right.
\end{equation*}

Thus we can apply Theorem \ref{lem uniqueness} to get a constant $T$ such that
\[u(x+2T,y)=v(-x,y),\ \ v(x+2T,y)=u(-x,y).\]
That is, $u$ and $v$ are symmetric with respect to the line $\{x=T\}$.

This symmetry implies that in the expansion \eqref{12.1}, $a=-b$. Now for any two solutions $(u_i,v_i)$ of \eqref{equation extension}, after a scaling, they have the expansions as in \eqref{fine asymptotics 1} and \eqref{fine asymptotics 2}, with
$a_1+b_1=a_2+b_2=0$. Thus by Theorem \ref{lem uniqueness}, $(u_1,v_1)$ and $(u_2,v_2)$ only differs by a translation in the $x$-direction. This completes the proof of Theorem \ref{main result}.
\end{proof}

\appendix
\section{Basic Facts about $s$-Lapalcian}
\setcounter{equation}{0}

We first present a mean value inequality for $L_a$-subharmonic function.
\begin{lem}\label{lem mean value inequality}
Let $u$ be a $L_a$-subharmonic function in $B_r\subset\R^{n+1}$ (centered at the origin), then
\[u(0)\leq C(n,a)r^{-n-1-a}\int_{B_r}y^au.\]
Here $C(n,a)$ is a constant depending only on $n$ and $a$.
\end{lem}
\begin{proof}
Direct calculation gives
\begin{eqnarray*}
\frac{d}{dr}\left(r^{-n-a}\int_{\partial B_r}y^au\right)&=&r^{-n-a}\int_{\partial B_r}y^a\frac{\partial u}{\partial r}\\
&=&r^{-n-a}\int_{B_r}\mbox{div}\left(y^a\nabla u\right)\\
&\geq&0.
\end{eqnarray*}
Thus $r^{-n-a}\int_{\partial B_r}y^au$ is non-decreasing in $r$. Integrating this in $r$ shows that
$r^{-n-1-a}\int_{B_r}y^au$ is also non-decreasing in $r$.
\end{proof}

By standard Moser's iteration we also have the following super bound
\begin{lem}\label{lem sup bound}
Let $u$ be a $L_a$-subharmonic function in $B_r\subset\R^{n+1}$ (centered at the origin), then
\[\sup_{B_{r/2}}u\leq C(n,a)\left(r^{-n-1-a}\int_{B_r}y^au^2\right)^{\frac{1}{2}}.\]
Here $C(n,a)$ is a constant depending only on $n$ and $a$.
\end{lem}

\begin{lem}\label{lem decay}
Let $M>0$ be fixed. Any $v\in H^1(B_1^+)\cap C(\overline{B_1^+})$ nonnegative solution to
\begin{equation*}
\left\{
\begin{aligned}
 &L_a v\geq 0, \ \mbox{in}\ B_1^+, \\
 &\partial^a_y v\geq Mv    \ \ \mbox{on}\ \partial B_1^+,
                          \end{aligned} \right.
\end{equation*}
satisfies
\[\sup_{\partial^0B_{1/2}^+}v\leq\frac{C(n)}{M}\int_{B_1^+}y^av.\]
\end{lem}
\begin{proof}
This is essentially \cite[Lemma 3.5]{TVZ 2}. We only need to note that, since
\[\partial_y^a v\geq 0   \ \ \mbox{on}\ \partial B_1^+,\]
the even extension of $v$ to $B_1$ is $L_a$-subharmonic (cf. \cite[Lemma 4.1]{C-S}). Then by Lemma \ref{lem sup bound},
\[\sup_{B_{2/3}^+}v\leq C(n)\int_{B_1^+}y^av.\qedhere\]
\end{proof}

\bigskip

\noindent
{\bf Acknowledgments:} The research of J. Wei is partially supported by NSERC of Canada.
Kelei Wang is supported by NSFC No. 11301522.


\bigskip




\begin{thebibliography}{50}
\small

\bibitem{apple} D. Applebaum, L\'{e}vy processes-from probability to finance and quantum groups, {\em Not. Am. Math. Soc.} 51 (2004), 1336-1347.



\bibitem{blwz}
H. Berestycki, T. Lin, J. Wei and C. Zhao, On phase-separation
model: asymptotics and qualitative properties, Arch. Ration. Mech. Anal. 208 (2013), no.1, 163-200.

\bibitem{BTWW}
H. Berestycki, S. Terracini, K. Wang and J. Wei, Existence and
stability of entire solutions of an elliptic system modeling phase
separation, Adv. Math. 243 (2013), 102-126.

\bibitem{C-K-L}
  L. A. Caffarelli, A. L. Karakhanyan and F. Lin, The geometry of
solutions to a segregation problem for non-divergence systems, {\it
J. Fixed Point Theory Appl.}  5 (2009), no.2, 319-351.


\bibitem{C-L 2}
 L. A. Caffarelli and F. Lin, Singularly perturbed elliptic
systems and multi-valued harmonic functions with free boundaries,
{\it Journal of the American Mathematical Society}  21(2008),
847-862.


\bibitem{C-S}
L. Caffarelli and L. Silvestre, An extension problem related to the fractional Lapalcian. {\em Communications in partial differential equations} 32 (2007), no. 8, 1245-1260.




\bibitem{C-T-V 3}
 M. Conti, S. Terracini and G. Verzini, Asymptotic estimates for
the spatial segregation of competitive systems, {\it Adv. Math.}
195(2005), no. 2, 524-560.

\bibitem{Hum} Humphries et al., Environmental context explains L\'{e}vy and Browian movement patterns of marine predators, {\em Nature} 465, June 2010, 1066-1069.


\bibitem{FKS}
E. Fabes, C. Kenig and R. Serapioni. The local regularity of solutions of degenerate elliptic equations.  {\em Communications in partial differential equations}  7 (1982), no. 1, 77-116.

\bibitem{F} A. Farina, Some symmetry results for entire solutions of an elliptic system arising in phase separation, {\em DCDS-A}, to appear.


\bibitem{FS} A. Farina and N. Soave, Monotonicity and 1-dimensional symmetry for
solutions of an elliptic system arising in
Bose-Einstein condensation, {\em Arch. Rat. Mech. Anal.} to appear.


\bibitem{FL} R.L. Frank and E. Lenzmann, Uniqueness of non-linear ground states for fractional Laplacians in $\R$, {\em Acta Math.} 210(2013), no.2, 261-318.

    \bibitem{FLS}  R.L. Frank, E. Lenzmann and L. Silvestre, Uniqueness of radial solutions for the fractional Laplacian, arxiv:1302.2652.





\bibitem{M}
Marijan Markovic, On harmonic functions and the hyperbolic metric, arXiv:1307.4006.


%\bibitem{P-Q-S}
%P. Pol\'{a}cik, P. Quittner, P. Souplet, Singularity and decay estimates in superlinear problems via Liouville-type
%theorems. I. Elliptic equations and systems. {\em Duke Math. J.} 139 (2007), no. 3, 555--579.


\bibitem{NTTV}
 B. Noris, H. Tavares, S. Terracini, G. Verzini, Uniform
H\"{o}lder bounds for nonlinear Schr\"{o}dinger systems with strong
competition, {\em Comm. Pure Appl. Math.} 63 (2010), 267--302


\bibitem{SY}
Schoen R, Yau S T. Lectures on differential geometry. Cambridge: International press, 1994.

\bibitem{RS} X. Ros-Oton and J. Serra, The  Dirichlet problem for fractional laplacian: regularity up to the boundary, arxiv:1207.5985v1

\bibitem{SZ} N.Soave and A. Zilio, Entire solutions with exponential growth for an elliptic system modeling phase-separation, {\em Nonlinearity} 27 (2014).

\bibitem{TX}
 J. Tan, J. Xiong, A Harnack inequality for fractional Laplace equations with lower order terms, {\em Disc. Cont. Dyna. Syst., A.} 31 (2011) 975-983.

\bibitem{TVZ}
S. Terracini, G. Verzini and A. Zilio, Uniform Holder bounds for strongly competing systems involving the square root of the Lapalcian, arXiv preprint arXiv:1211.6087 (2012).

\bibitem{TVZ 2}
S. Terracini, G. Verzini and A. Zilio, Uniform H\" older regularity with small exponent in competition-fractional diffusion systems, arXiv preprint arXiv:1303.6079.

\bibitem{VZ} G. Verzini and A. Zilio, Strong competition versus fractional diffusion: the case of Lotka-Volterra interaction, arxiv:1310.7355v1

 \bibitem{W1}
K. Wang, On the De Giorgi type conjecture for an elliptic system modeling phase separation, to appear in {\em Comm. Partial Differential Equations (2014)}.

\bibitem{W2}
K. Wang, Harmonic approximation and improvement of flatness in a singular perturbation problem, arXiv:1401.3517.
\end{thebibliography}
\end{document}